\documentclass[12 pt]{amsart}
\usepackage{amssymb, amsmath, amsfonts, amsthm, graphics,mathrsfs, mathtools}
\usepackage[usenames, dvipsnames]{xcolor}
\usepackage[hmargin=1 in, vmargin = 1 in]{geometry}
\usepackage{tikz-cd}
\usetikzlibrary{matrix, calc, arrows} 
\usepackage{hyperref}
\usepackage[all]{xy}
\usepackage{ytableau}

\definecolor{darkblue}{rgb}{0.0,0,0.7}

\newcommand{\newword}[1]{\textcolor{darkblue}{\textbf{\emph{#1}}}}

\newcommand{\domzz}
{\raisebox{-0.25\height}{ \begin{tikzpicture}[x=0.75 em,y=0.75 em]
\draw[step=1,gray, thin] (0,0) grid (1,2);
\draw[color=black, thick](0,0)rectangle(1,2);
\draw[thick,rounded corners,color=blue] (0,1.5)--(0.5,1.5)--(0.5,2);
\draw[thick,rounded corners,color=blue] (0,0.5)--(0.5,0.5)--(0.5,1.5)--(1,1.5);
\draw[thick,rounded corners,color=blue] (0.5,0)--(0.5,0.5)--(1,0.5);
\end{tikzpicture}}}

\newcommand{\domzp}
{\raisebox{-0.25\height}{ \begin{tikzpicture}[x=0.75 em,y=0.75 em]
\draw[step=1,gray, thin] (0,0) grid (1,2);
\draw[color=black, thick](0,0)rectangle(1,2);
\draw[thick,rounded corners,color=blue] (0,1.5)--(0.5,1.5)--(0.5,2);
\draw[thick,rounded corners,color=blue] (0,0.5)--(1,0.5);
\draw[thick,rounded corners,color=blue] (0.5,0)--(0.5,1.5)--(1,1.5);
\end{tikzpicture}}}

\newcommand{\dompz}
{\raisebox{-0.25\height}{ \begin{tikzpicture}[x=0.75 em,y=0.75 em]
\draw[step=1,gray, thin] (0,0) grid (1,2);
\draw[color=black, thick](0,0)rectangle(1,2);
\draw[thick,rounded corners,color=blue] (0,1.5)--(1,1.5);
\draw[thick,rounded corners,color=blue] (0,0.5)--(0.5,0.5)--(0.5,2);
\draw[thick,rounded corners,color=blue] (0.5,0)--(0.5,0.5)--(1,0.5);
\end{tikzpicture}}}

\newcommand{\dompp}
{\raisebox{-0.25\height}{ \begin{tikzpicture}[x=0.75 em,y=0.75 em]
\draw[step=1,gray, thin] (0,0) grid (1,2);
\draw[color=black, thick](0,0)rectangle(1,2);
\draw[thick,rounded corners,color=blue] (0,1.5)--(1,1.5);
\draw[thick,rounded corners,color=blue] (0,0.5)--(1,0.5);
\draw[thick,rounded corners,color=blue] (0.5,0)--(0.5,2);
\end{tikzpicture}}}

\DeclareMathOperator{\Spec}{Spec}

\DeclareMathOperator\codim{codim}

\newcommand{\Flags}{\mathsf{Flags}}

\newcommand{\id}{\mathrm{id}}

\newcommand{\CC}{\mathbb{C}}

\newcommand{\ZZ}{\mathbb{Z}}

\newcommand{\bx}{\boldsymbol{x}}
\newcommand{\by}{\boldsymbol{y}}

\newcommand{\fS}{\mathfrak{S}}
\newcommand{\fG}{\mathfrak{G}}
\newcommand{\fR}{\mathfrak{R}}
\newcommand{\fCM}{\mathfrak{CM}}

\newcommand{\inv}{\mathrm{inv}}
\newcommand{\maj}{\mathrm{maj}}
\newcommand{\raj}{\mathrm{raj}}
\newcommand{\rajc}{\raj{\mathrm{code}}}
\newcommand{\invc}{\inv{\mathrm{code}}}

\DeclareMathOperator{\im}{\mathrm{im}}

\newcommand{\word}{{\sf word}}

\newcommand{\pipes}{{\sf Pipes}}
\newcommand{\Rothe}{{\sf RD}}

\newcommand{\rN}{\widehat{\partial}}

\newtheorem{Theorem}{Theorem}[section]

\newtheorem{cor}[Theorem]{Corollary}
\newtheorem{corollary}[Theorem]{Corollary}

\newtheorem{definition}[Theorem]{Definition}
\newtheorem{proposition}[Theorem]{Proposition}
\newtheorem{prop}[Theorem]{Proposition}

\newtheorem{theorem}[Theorem]{Theorem}

\newtheorem{lem}[Theorem]{Lemma}
\newtheorem{lemma}[Theorem]{Lemma}
\newtheorem*{claim}{Claim}

\theoremstyle{remark}
\newtheorem{examplex}[Theorem]{Example}
\newenvironment{example}
  {\pushQED{\qed}\examplex}
  {\popQED\endexamplex}

\newtheorem{remark}[Theorem]{Remark}
\newtheorem{rem}[Theorem]{Remark}

\title{Castelnuovo--Mumford regularity of matrix Schubert varieties}
\author{Oliver Pechenik}
\address[OP]{Department of Combinatorics \& Optimization, University of Waterloo, Waterloo ON N2L3G1}
\email{oliver.pechenik@uwaterloo.ca}
\author{David E Speyer}
\address[DES]{Department of Mathematics, University of Michigan, Ann Arbor MI 48109}
\email{speyer@umich.edu}
\author{Anna Weigandt}
\address[AW]{Department of Mathematics, Massachusetts Institute of Technology, Cambridge MA 02139}
\email{weigandt@mit.edu}

\begin{document}

\begin{abstract}
    Matrix Schubert varieties are affine varieties arising in the Schubert calculus of the complete flag variety. We give a formula for the Castelnuovo--Mumford regularity of matrix Schubert varieties, answering a question of Jenna Rajchgot. We follow her proposed strategy of studying the highest-degree homogeneous parts of Grothendieck polynomials, which we call Castelnuovo--Mumford polynomials. In addition to the regularity formula, we obtain formulas for the degrees of all Castelnuovo--Mumford polynomials and for their leading terms, as well as a complete description of when two Castelnuovo--Mumford polynomials agree up to scalar multiple. 
    The degree of the Grothendieck polynomial is a new permutation statistic which we call the Rajchgot index; we develop the properties of Rajchgot index and relate it to major index and to weak order.
\end{abstract}

\maketitle

\section{Introduction}

The \emph{flag variety} $\Flags_n$, the parameter space for complete flags of nested vector subspaces of $\CC^n$, has a complex cell decomposition given by its \emph{Schubert varieties}. The geometry and combinatorics of this cell decomposition are of central importance in Schubert calculus. 
These Schubert varieties are closely related to certain generalized determinantal varieties $X_w$ of $n \times n$ matrices called  \emph{matrix Schubert varieties} (see~\cite{Fulton} and Section~\ref{sec:matrixSchubertvarieties} for the definition).
 These varieties have been heavily studied from various perspectives (see, e.g., \cite{Escobar.Meszaros, Fink.Rajchgot.Sullivant, Fulton, Hamaker.Pechenik.Weigandt, Hsiao, Knutson.Miller, Knutson.Miller.Yong, Weigandt.Yong}).  It is natural to desire a measure of the algebraic complexity of matrix Schubert varieties. One such measure is the \emph{Castelnuovo--Mumford regularity} of $X_w$, a commutative-algebraic invariant determining the extent to which the defining ideal of $X_w$ can be resolved by low-degree polynomials.

Jenna Rajchgot (cf.\ \cite{RRRSW}) noted that, since matrix Schubert varieties are Cohen--Macaulay \cite{Fulton, Knutson.Miller, Ramanathan}, the regularity of $X_w$ is given by the difference between the highest-degree and lowest-degree homogeneous parts of the \emph{$K$-polynomial} for $X_w$. These particular $K$-polynomials have been much studied. They were introduced by Lascoux and Sch\"utzenberger \cite{Lascoux.Schutzenberger}, under the name of \emph{Grothendieck polynomials} $\fG_w(\bx)$, as polynomial representatives for structure sheaf classes in the $K$-theoretic Schubert calculus of $\Flags_n$ (see also, \cite{Fulton.Lascoux}). Grothendieck polynomials are inhomogeneous polynomials $\fG_w(\bx)$ in $n$ variables $\bx = x_1, x_2, \ldots, x_n$, indexed by permutations $w$ in the symmetric group $S_n$. Later, Knutson and Miller \cite{Knutson.Miller} showed that Grothendieck polynomials coincide, up to convention choices, with the $K$-polynomials of matrix Schubert varieties.

The lowest-degree homogeneous part of $\fG_w(\bx)$ is the \emph{Schubert polynomial} $\fS_w(\bx)$ \cite{Lascoux.Schutzenberger:Schubert}; Schubert polynomials are well-understood from a combinatorial perspective, and the degree of $\fS_w(\bx)$ equals the codimension of $X_w$ or equivalently the Coxeter length $\inv(w)$ of the permutation $w$. Hence, determining the regularity of $X_w$ reduces to answering the following question of Rajchgot:
\begin{quotation}
``What is the degree of a Grothendieck polynomial?" 
\end{quotation}

In light of these observations, we term the highest-degree part of $(-1)^{\deg \fG_w(\bx) - \inv(w)} \fG_w(\bx)$ the \newword{Castelnuovo--Mumford polynomial} and write it $\fCM_w(\bx)$.
(The power of $-1$ makes $\fCM_w(\bx)$ have positive coefficients.)
 The goal of this paper is to answer Rajchgot's question by understanding these homogeneous polynomials and in particular their degrees, thereby obtaining a formula for the Castelnuovo--Mumford regularity of $X_w$. In the special case of \emph{symmetric} Grothendieck polynomials, corresponding to \emph{Grassmannian permutations} $w$ and Schubert varieties in a complex Grassmannian, a formula for the degree of $\fCM_w(\bx)$ was recently given in \cite{RRRSW}. 
 Formulas for \emph{vexillary} and \emph{$1432$-avoiding} permutations (and related objects) appear in forthcoming work \cite{Rajchgot.Robichaux.Weigandt}.
 Our first main result is a degree formula for arbitrary $\fCM_w(\bx)$, answering Rajchgot's question in full generality. 

Write a permutation $w \in S_n$ in one-line notation as $w(1) w(2) \cdots w(n)$. For each $k$, find an increasing subsequence of $w(k) w(k+1) \cdots w(n)$ containing $w(k)$ and of greatest length among such subsequences. Let $r_k$ be the number of terms from $w(k) w(k+1) \cdots w(n)$ omitted from this subsequence. We call the sequence $(r_1, \dots, r_n) = \rajc(w)$ the \newword{Rajchgot code} of $w$ and its sum $\raj(w)$ the \newword{Rajchgot index} of $w$.

\begin{Theorem}\label{thm:degree}
For $w \in S_n$, we have $\deg \fCM_w(\bx) = \raj(w)$. Moreover, for any term order satisfying $x_1 < x_2 < \cdots < x_n$, the leading term of $\fCM_w(\bx)$ is a scalar  multiple of the monomial $\bx^{\rajc(w)} =  x_1^{r_1} x_2^{r_2} \cdots x_n^{r_n}$.  

In particular, the Castelnuovo--Mumford regularity of the matrix Schubert variety $X_w$ is $\raj(w) - \inv(w)$.
\end{Theorem}

\begin{example}\label{ex:main}
Consider the permutation $w = 293417568 \in S_9$. A longest increasing subsequence starting from $2$ is $2\bullet 34\bullet\bullet568$, which omits three terms, so $r_1 = 3$. In full,  \[\rajc(w) = (r_1, r_2, \ldots, r_9) = (3,7,2,2,1,2,0,0,0).\] 
Hence, by Theorem~\ref{thm:degree}, the leading term of $\fCM_{w}(\bx)$ is a scalar multiple of the monomial $x_1^3 x_2^7 x_3^2 x_4^2 x_5 x_6^2$ and the degree of $\fCM_{w}(\bx)$ is $\raj(w) = 3+7+2+2+1+2+0+0+0 = 17$. Since $\inv(w) = 12$, it follows that the Castelnuovo--Mumford regularity of the matrix Schubert variety $X_{w}$ is $\raj(w) - \inv(w) = 17-12 = 5$.
\end{example}

\begin{proof}[Proof of Theorem~\ref{thm:degree}]
The verification that $\deg \fCM_w(\bx) = \raj(w)$ is Theorem~\ref{thm:raj=deg}.
The claim about the leading term follows from Theorems~\ref{thm:exponentbound} and~\ref{thm:exponentachieved}.
The consequence for Castelnuovo--Mumford regularity is explained in Corollary~\ref{cor:reg}.
\end{proof}

\begin{remark}
The theorems in this introductory section are stated in what the authors hope is the clearest order to explain the results, not in the order of their proof. 
In the main body of the paper, the results appear in their logical order, and we reference those results from the introduction, as in the proof above. 
The reader can thus see that the remaining sections of the paper have no forward citations, and the results in this section are not referenced in the proofs of the following sections, so there is no circularity.
\end{remark}

Our remaining results explore the combinatorics of Castelnuovo--Mumford polynomials and the associated permutation statistics.

 While Schubert polynomials are all distinct and have distinct leading monomials, we observe that many Castelnuovo--Mumford polynomials differ only by a scalar multiple. In fact, we will show that $\fCM_u(\bx)$ and $\fCM_v(\bx)$ differ by a scalar precisely if $\rajc(u) = \rajc(v)$. This phenomenon is best understood in the context of \emph{double} Castelnuovo--Mumford polynomials, as we now explain.
The \emph{double Grothendieck polynomials} are certain polynomials $\fG_w(x_1,\ldots, x_n; y_1, \ldots, y_n)$ in $2n$ variables, also indexed by $w \in S_n$. They represent Schubert classes in the torus-equivariant $K$-theory of $\Flags_n$, and obey the relations $\fG_w(\bx; 0) = \fG_w(\bx)$ and $\fG_w(\bx;\by) = \fG_{w^{-1}}(\by; \bx)$. 
We define the \newword{double Castelnuovo--Mumford polynomial} $\fCM_w(\bx;  \by)$ to be the highest-degree part of $\fG_w(\bx; \by)$.
We will show (Corollary~\ref{cor:high_degree_xy}) that $\fG_w(\bx, \by)$ has terms whose $\bx$-degree and $\by$-degree are simultaneously maximal, so $\fCM_w(\bx, \by)$ is homogeneous in both $\bx$ and $\by$.

Remarkably, we find that double Castelnuovo--Mumford polynomials factor as a polynomial in $\bx$ times a polynomial in $\by$. We identify a special family of single Castelnuovo--Mumford polynomials, which we refer to as the \newword{Rajchgot polynomials} $\fR_{\pi}(\bx)$, indexed by set partitions of $\{1, \dots, n\}$. For each $w \in S_n$, we associate a set partition $\pi(w)$ so that the following holds.

\begin{Theorem} \label{thm:main}
Double Castelnuovo--Mumford polynomials factor into Rajchgot polynomials as \[\fCM_{w}(\bx;  \by) = \fR_{\pi(w)}(\bx) \fR_{\pi(w^{-1})}(\by).\] 
Moreover, for any term order satisfying 
$x_1 < \cdots < x_n$ and $y_1 < \cdots < y_n,$ the leading term of the double Castelnuovo--Mumford polynomial $\fCM_w(\bx;\by)$ is exactly $\bx^{\rajc(w)} \by^{\rajc(w^{-1})}$. 
In particular, the single Castelnuovo--Mumford polynomial $\fCM_w(\bx)$ is $\fR_{\pi(w^{-1})}(1, \dots, 1) \fR_{\pi(w)}(\bx)$ and has leading term $\fR_{\pi(w^{-1})}(1, \dots, 1) \bx^{\rajc(w)}$.
\end{Theorem}

\begin{proof}
The factorization is established in Theorem~\ref{thm:factorization}. 
The claim about the leading monomial follows from Theorems~\ref{thm:exponentbound} and~\ref{thm:exponentachieved}; the fact that this monomial has coefficient $1$ is Theorem~\ref{thm:coeff1}.
The remaining statements follow from the equality $\fCM_w(\bx) = \fCM_w(\bx; 1,1,\ldots,1)$.
\end{proof}

In particular, Theorem~\ref{thm:main} shows that, up to scalar multiple, the number of distinct Castelnuovo--Mumford polynomials for $w \in S_n$ is not $n!$, but rather the number of set partitions of $n$, which is also known as the  $n$-th \emph{Bell number}.

 The Rajchgot index is related to the classical \emph{major index} statistic. In particular, we prove the following:
 
 \begin{theorem}\label{thm:main_mm}
 	For all $w \in S_n$, we have 
 	\[
 	\raj(w) = \max \{\maj(v) : v \leq_R w \} = \max \{\maj(u^{-1}) : u \leq_L w \} = \deg \fCM_w(\bx),
 	\]
 	 where $\leq_L$ and $\leq_R$ denote the \emph{left and right weak orders}, respectively.  
 	 \end{theorem}
	 
To the best of our knowledge, none of the equalities in Theorem~\ref{thm:main_mm} have been observed previously.
 
 \begin{proof}
 The equality between $\raj(w)$ and $\deg \fCM_w(\bx)$ was proved above; the formulas in terms of  the major index statistic appear in Theorem~\ref{thm:raj=mm}.
 \end{proof}

 As an application of these ideas, we also determine the maximum Castelnuovo--Mumford regularity for matrix Schubert varieties $X_w$ with $w \in S_n$.
\begin{theorem} \label{thm:MaxRegIntro}
Let $n$ be a positive integer and define $k$ by $\binom{k}{2} \leq n \leq \binom{k+1}{2}$; if $n$ is a triangular number, then we may choose either value for $k$.
For matrix Schubert varieties $X_w$ with $w \in S_n$, the largest Castelnuovo--Mumford regularity is $\binom{n+1}{2} - kn + \binom{k+1}{3}$. 
\end{theorem}

In Theorem~\ref{thm:MaxReg}, we will characterize the permutations that achieve this maximum.

\begin{proof}
In Theorem~\ref{thm:MaxReg}, we will show that $\raj(w) - \inv(w) \leq \binom{n+1}{2} - kn + \binom{k+1}{3}$ (and find permutations that achieve equality).
As noted above, $\raj(w) - \inv(w)$ is the Castelnuovo--Mumford regularity of $X_w$.
\end{proof}
 
 This paper is organized as follows. In Section~\ref{sec:background}, we review necessary background. In Section~\ref{sec:rajchgot}, we begin to develop the combinatorics of Rajchgot index and introduce \newword{fireworks permutations}. In Section~\ref{sec:fireworks}, we introduce the key tools of the \newword{blob diagram} and the \newword{fireworks map}.  In Section~\ref{sec:degrees}, we apply these results to establish most of Theorem~\ref{thm:degree}; we show that $\deg \fCM_w(\bx) = \raj(w)$ and therefore that the Castelnuovo--Mumford regularity of $X_w$ is $\raj(w) - \inv(w)$. Section~\ref{sec:rajchgotpolynomials} introduces Rajchgot polynomials, establishes the factorization result for double Castelnuovo--Mumford polynomials $\fCM_w(\bx; \by)$ of Theorem~\ref{thm:main} and shows that the leading term of $\fCM_w(\bx, \by)$ is at most as large as predicted by Theorem~\ref{thm:main}. In Section~\ref{sec:pushingPluses}, we complete our proofs by constructing a pipe dream whose degree is $\bx^{\rajc(w)} \by^{\rajc(w^{-1})}$.

\section{Background}
\label{sec:background}

\subsection{Permutations}

Let $[n] \coloneqq \{1, 2, \dots, n\}$.
Let $S_n$ denote the symmetric group of permutations of $[n]$. We consider $w \in S_n$ as a map $w \colon [n] \to [n]$ and write $w$ in \newword{one-line notation} as the string $w(1) w(2) \cdots w(n)$. We will often write $w_i \coloneqq w(i)$.  We identify $w \in S_n$ with the permutation matrix having a $1$ in each position $(i, w_i)$ and $0$s elsewhere. We will also often identify $S_{n-1}$ with the subgroup $\{ w \in S_n : w_n = n \} \subset S_n$ fixing $n$. We write $\id$ for the identity permutation $12 \cdots n$ and $w_0$ for the reverse permutation $n(n-1) \cdots 1$.

Let $s_i \coloneqq (i \; i+1)$ denote the simple transposition that exchanges $i$ and $i+1$, and recall that $s_1, \dots, s_{n-1} $ together generate $S_n$.  We note that, because we write a permutation $w$ in one-line notation as $w(1) w(2) \cdots w(n)$, multiplying $w$ on the \emph{left} by $s_k$ switches the \emph{values} $k$ and $k+1$ and multiplying on the \emph{right} by $s_k$ switches the values in positions $k$ and $k+1$.

An \newword{inversion} of $w \in S_n$ is a pair $i,j \in [n]$ such that $i<j$ and $w_i > w_j$. We write $\inv(w)$ for the number of inversions in $w$ and call this quantity the \newword{Coxeter length} of $w$. Note that $\inv(w)$ is the length of the shortest expression for $w$ as a product of the generators $s_i$. A factorization $w = s_{i_1} \cdots s_{i_{\inv(w)}}$ is called a \newword{reduced expression} for $w$, and the sequence of subscripts $i_1 \cdots i_{\inv(w)}$ is called a \newword{reduced word} for $w$.

We will have need of four different partial orders on the set $S_n$. If $w = uv$ with $\inv(w) = \inv(u) + \inv(v)$, then we say that $v \leq_L w$ and $u \leq_R w$; the relations $\leq_L$ and $\leq_R$ are known as \newword{left} and \newword{right weak order}, respectively. We write $\leq_{LR}$ for the partial order obtained as the transitive closure of the union of left and right weak orders and call this the \newword{two-sided weak order} (see \cite{Petersen}). For $u \leq_R v$, we write $[u,v]_R$ for the interval from $u$ to $v$ in right weak order. Similarly, we define the notations $[u,v]_L$ and $[u,v]_{LR}$ in the analogous ways. Finally, we define \newword{Bruhat order} on $S_n$ by $v \leq w$ if some reduced word for $v$ is a substring of some reduced word for $w$. The various weak orders are ``weak'' in the sense that they are proper subrelations of Bruhat order.

A \newword{descent} of $w \in S_n$ is a value $i$ such that $w_i > w_{i+1}$. Equivalently, $i$ is a descent of $w$ if $\inv(w) > \inv(ws_i)$. A permutation $w$ is \newword{Grassmannian} if it has at most one descent.

The \newword{$0$-Hecke monoid} $\mathcal{H}_n$ is the free monoid on generators $\tau_1, \dots, \tau_{n-1}$ subject to the ``idempotent braid relations''
\[
\tau_i^2 = \tau_i, \quad \tau_i \tau_j = \tau_j \tau_i,  \; \text{for $j \neq i \pm 1$,} \quad \text{and} \quad  \tau_i  \tau_{i+1}  \tau_i =  \tau_{i+1}  \tau_i  \tau_{i+1}.
\]
There is a natural action of $\mathcal{H}_n$ on $S_n$ induced by
\[
\tau_i \ast w \coloneqq \begin{cases}
	s_i w, & \text{if} \; \inv(s_i w) > \inv(w);  \\
	w, & \text{otherwise.}
\end{cases}
\]
For every permutation $w$ in $S_n$, there is a unique element $\bar{w}$ in $\mathcal{H}_n$ with $\bar{w} \ast \id = w$; for example, $\bar{s}_i = \tau_i$. 
We define the \newword{Demazure product} on $S_n$ to be the binary operation $u \ast v$ given by $u \ast v = \bar{u} \ast \bar{v} \ast \id$. 

The \newword{graph} of the permutation $w \in S_n$ is obtained by plotting bullets $\bullet$ in the $n\times n$ grid in positions $(i,w_i)$ for $i\in [n]$ (in matrix coordinates).  
The \newword{Rothe diagram} $\Rothe(w)$ of $w$ is constructed from its graph as follows. From each $\bullet$, fire a laser beam directly to the right and another straight down. The cells of the $n \times n$ grid that are hit by no laser beam are the Rothe diagram $\Rothe(w)$. It is not hard to see that the number of cells in $\Rothe(w)$ is $\inv(w)$. Write $\ell_i$ for the number of cells in row $i$. We call the sequence $\invc(w) \coloneqq (\ell_1, \dots, \ell_n)$ the \newword{inv code} of $w$. (This is also often referred to as the \emph{Lehmer code}, but we won't use this terminology.) Note that $\inv(w)$ is the sum of $\invc(w)$.

\begin{example}\label{ex:Rothe}
	Suppose $w = 42153 \in S_5$. Then the Rothe diagram of $w$ is the gray cells of the diagram
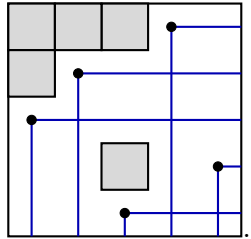
\[
  \begin{tikzpicture}[x=1.5em,y=1.5em]
      \draw[color=black, thick](0,1)rectangle(5,6);
     \filldraw[color=black, fill=gray!30, thick](0,5)rectangle(1,6);
     \filldraw[color=black, fill=gray!30, thick](0,4)rectangle(1,5);
      \filldraw[color=black, fill=gray!30, thick](2,5)rectangle(3,6);
     \filldraw[color=black, fill=gray!30, thick](1,5)rectangle(2,6);
     \filldraw[color=black, fill=gray!30, thick](2,2)rectangle(3,3);
     \draw[thick,darkblue] (5,5.5)--(3.5,5.5)--(3.5,1);
     \draw[thick,darkblue] (5,4.5)--(1.5,4.5)--(1.5,1);
     \draw[thick,darkblue] (5,3.5)--(.5,3.5)--(.5,1);
     \draw[thick,darkblue] (5,2.5)--(4.5,2.5)--(4.5,1);
     \draw[thick,darkblue] (5,1.5)--(2.5,1.5)--(2.5,1);
     \filldraw [black](3.5,5.5)circle(.1);
     \filldraw [black](1.5,4.5)circle(.1);
     \filldraw [black](.5,3.5)circle(.1);
     \filldraw [black](4.5,2.5)circle(.1);
     \filldraw [black](2.5,1.5)circle(.1);
     \end{tikzpicture}.
\]
Hence we have $\invc(w) = (3,1,0,1,0)$ and $\inv(w) = 5$.
\end{example}

Suppose $v \in S_k$ and $w \in S_n$ with $k \leq n$. We say $w$ \newword{contains} $v$ if there is a subsequence $w_{i_1}, w_{i_2}, \dots, w_{i_k}$ such that we have $v_p < v_q \Longleftrightarrow w_{i_p} < w_{i_q}$ for all $1 \leq p,q \leq k$. If $w$ does not contain $v$, we say $w$ \newword{avoids}  $v$. A \newword{dominant} permutation is one that avoids $132$.
 
 \subsection{Matrix Schubert varieties}\label{sec:matrixSchubertvarieties}
The \newword{matrix Schubert variety} $X_w$ is an affine variety cut out by certain determinants.  
Let $Z = (z_{ij})_{1 \leq i,j \leq n}$ be a matrix of distinct indeterminants. Then $X_w$ is a subvariety of the $n^2$-dimensional affine space $\Spec \mathbb{C}[Z]$. 

Consider the Rothe diagram $\Rothe(w)$. For each cell $(i,j)$ of $\Rothe(w)$, let $r_{i,j}$ denote the number of $1$s appearing in the permutation matrix $w$ northwest of the cell $(i,j)$. 
Let $I_w$ be the ideal generated by, for each $(i,j)$, the $(r_{i,j} + 1) \times (r_{i,j} + 1)$ minors of the matrix northwest of $(i,j)$. 
(If $r_{i,j} = \min(i,j)$, so that no such minor fits in the matrix, then we obtain an empty list of generators in this case.)
The matrix Schubert variety $X_w$ is the subvariety of $n \times n$ matrices defined by the ideal $I_w$. By work of Fulton \cite{Fulton}, the ideal $I$ is prime, so $X_w$ is a reduced and irreducible affine variety; a $n \times n$ matrix $A$ lies in $X_w$ if the rank of each northwest submatrix of $A$ is less than or equal to the rank of the same submatrix of $w$.

\subsection{Castelnuovo--Mumford regularity} \label{sec:CM-regularity}

Let $R \coloneqq \mathbb{C}[Z]$ be a standard-graded polynomial ring, so $\deg(z_{ij}) = 1$, and let $I \subseteq R$ be a homogeneous ideal. We write $R(-i)$ for $R$ with all degrees shifted by $i$. A \newword{free resolution} of $R/I$ is a diagram of graded $R$-modules
\[
0 \to \bigoplus_{i \in \mathbb{Z}} R(-i)^{b^k_i} \to \dots \to \bigoplus_{i \in \mathbb{Z}} R(-i)^{b^0_i} \to R/I \to 0
\] that is \newword{exact}, that is, such that the image of each map is the kernel of the next. By Hilbert's Basis and Syzygy Theorems, there always exists such a free resolution with $k \leq n^2$. Up to isomorphism there is a unique free resolution simultaneously minimizing all $b_i^j$. We call this the \newword{minimal free resolution} of $R/I$. In this case, the dimensions $b_i^j$ are invariants of $R/I$. The \newword{Castelnuovo--Mumford regularity} ${\rm reg}(R/I)$ of $R/I$ is the greatest $i - j$ such that $b_i^j \neq 0$. Conflating affine varieties with their coordinate rings, we also refer to this number as the Castelnuovo--Mumford regularity of $\Spec R/I$. In the case that $R/I$ is Cohen--Macaulay, it is known that the projective dimension of $R/I$ equals the height of the ideal $I$ as well as the codimension of $\Spec R/I$ in $\Spec R$.

Write $(R/I)_a$ for the degree $a$ piece of $R/I$.  The \newword{Hilbert series} of $R/I$ is the formal power series \[H(R/I;t)=\sum_{a\in \mathbb N} \dim_{\mathbb C}(R/I)_a t^a.\]
If we write the Hilbert series as a rational expression \[H(R/I;t)=\frac{K(R/I;t)}{(1-t)^{n^2}}\] the numerator $K(R/I;t)$ is the \newword{K-polynomial} of $R/I$.

The \newword{height} ${\rm ht}(I)$ of a prime ideal $I$ is the maximum $k$ so that there exists a nested chain of prime ideals \[I_0\subsetneq I_1\subsetneq \cdots \subsetneq I_k=I.\] If $I$ is prime, ${\rm ht}(I)$ is the codimension of $\Spec R/I$ in $\Spec R$.

The following lemma is well known to experts, see e.g.\ \cite{Bruns.Herzog,Benedetti.Varbaro,RRRSW}.
\begin{lemma}\label{lem:Benedetti}
Suppose $R/I$ is Cohen--Macaulay.  Then ${\rm reg}(R/I)= \deg(K(R/I;t))-{\rm ht}(I)$.
\end{lemma}

Matrix Schubert varieties are Cohen--Macaulay \cite{Fulton} and the codimension of $X_w$ is the inversion statistic $\inv(w)$. Hence, computing the regularity of matrix Schubert varieties amounts to finding the degree of their K-polynomials.

\subsection{Schubert and Grothendieck Polynomials} \label{GrothBackground}

Consider the polynomial ring $\mathbb Z[\bx;\by]=\mathbb Z[x_1,x_2,\ldots,x_n;y_1,y_2,\ldots,y_n]$.  There is a natural action of $S_n$ on $\mathbb Z[\bx;\by]$ defined by \[w\cdot f=f(x_{w_1},x_{w_2},\ldots,x_{w_n};y_1,\ldots,y_n).\] Given $f\in \mathbb Z[\bx;\by]$ and $1 \leq i < n$, we define the action of the \newword{divided difference operator} $\partial_{i}$ by \[\partial_i(f)=\frac{f-s_{i}\cdot f}{x_i-x_{i+1}}.\] Note that $\partial_i(f)$ is a polynomial, symmetric under interchanging the variables $x_i$ and $x_{i+1}$.
Let $\overline{\partial_i}$ be defined by 
\[\overline{\partial_i}(f)=\partial_i((1-x_{i+1})f).\]
The operators $\overline{\partial_i}$ satisfy the braid relations 
\[
\overline{\partial}_i \overline{\partial}_j = \overline{\partial}_j \overline{\partial}_i,  \; \text{for $j \neq i \pm 1$,} \quad \text{and} \quad \overline \partial_i \overline \partial_{i+1} \overline \partial_i = \overline \partial_{i+1} \overline \partial_i \overline \partial_{i+1}.
\]
and the idempotent relation 
\[ 
\overline{\partial_i}^2 = \overline{\partial_i}.
\]
 The double Grothendieck polynomials can be defined by the recursion:
\[ \overline{\partial_i} \fG_w(\bx, \by) = \begin{cases} \fG_{w s_i}(\bx, \by) & ws_i <_R w \\ \fG_{w}(\bx, \by) & ws_i >_R w \\ \end{cases} \]
together with the initial condition
\[ \fG_{w_0}(\bx, \by) = \prod_{i+j \leq n} (x_i+y_j - x_i y_j) . \] 
We obtain the \newword{(single) Grothendieck polynomials} by specializing $\by$ to 0, that is to say, \ \[\fG_w(\bx) \coloneqq \fG_w(\bx;\mathbf 0).\]

The \newword{double Schubert polynomial} $\fS_w(\bx, \by)$ is the lowest-degree homogeneous part of the double Grothendieck polynomial $\fG_w(\bx, \by)$ substituting $y_j \mapsto -y_j$, while the \newword{(single) Schubert polynomial} $\fS_w(\bx)$ is the lowest-degree homogeneous part of $\fG_w(\bx)$. 
The degree of the Schubert polynomial $\fS_w(\bx)$ is $\inv(w)$.

We now recall an explicit combinatorial formula for (double) Schubert polynomials and (double) Grothendieck polynomials.  A \newword{pipe dream} is a subset $P$ of the cells in the strictly upper left triangular part of the $n\times n$ grid, i.e.\ \[P\subseteq \{(i,j):1<i+j\leq n\}.\]  We represent this subset pictorially by placing a \newword{crossing tile} \begin{tikzpicture}[x=1em,y=1em]
	\draw[color=black, thick](0,1)rectangle(1,2);
	\draw[thick,rounded corners, color=darkblue] (0,1.5)--(1,1.5);
	\draw[thick,rounded corners,color=darkblue] (.5,1)--(.5,2);
	\end{tikzpicture}
in each cell of $P$ and \newword{bumping tiles}
\begin{tikzpicture}[x=1em,y=1em]
	\draw[color=black, thick](0,1)rectangle(1,2);
	\draw[thick,rounded corners,color=darkblue] (.5,1)--(.5,1.5)--(1,1.5);
	\draw[thick,rounded corners,color=darkblue] (.5,2)--(.5,1.5)--(0,1.5);
	\end{tikzpicture}
 in the other cells.

If there is a crossing tile in cell $(i,j)$, we associate to it the simple transposition $s_{i+j-1}$.  
We then associate a reading word $\word(P)$ to $P$ by reading these simple transpositions within rows from right to left, working from the top row downwards.  We say $P$ is a pipe dream for $w$ if $w$ is the Demazure product of $\word(P)$.
Write $\pipes(w)$ for the set of pipe dreams for $w$. We say $P \in \pipes(w)$ is \newword{reduced} if $\word(P)$ is a reduced word for $w$ and write $\pipes_0(w)$ for the subset of such reduced pipe dreams $P$.

The following theorem first appeared in this form in \cite{Knutson.Miller}; however, special cases and essentially equivalent formulations were known earlier, e.g.\ in \cite{Bergeron.Billey, Fomin.Kirillov}.	
\begin{theorem}[{\cite{Knutson.Miller}}]\label{thm:pipedreams}
For any $w \in S_n$, we have
\[ \fG_w(\bx)=\sum_{P\in \pipes(w)} (-1)^{|P|-\inv(w)} \prod_{(i,j)\in P}  x_i  \]
\[ \fG_w(\bx;\by)=\sum_{P\in \pipes(w)}(-1)^{|P|-\inv(w)} \prod_{(i,j)\in P} (x_i+y_j-x_iy_j) \]
and 
\[ \fS_w(\bx;\by)=\sum_{P\in \pipes_0(w)}\prod_{(i,j)\in P} (x_i-y_j)
\]
\end{theorem}

\begin{example}
Pictured below are the reduced pipe dreams for $w=42153$.
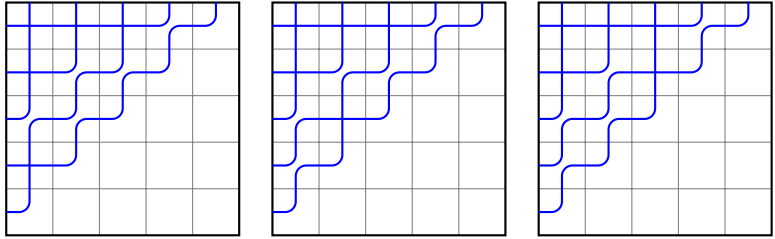
\[\begin{tikzpicture}[x=1.5em,y=1.5em]
\draw[step=1,gray, thin] (0,0) grid (5,5);
\draw[color=black, thick](0,0)rectangle(5,5);
\draw[thick,rounded corners,color=blue](0,4.5)--(3.5,4.5)--(3.5,5);
\draw[thick,rounded corners,color=blue] (0,3.5)--(1.5,3.5)--(1.5,5);
\draw[thick,rounded corners,color=blue] (0,2.5)--(.5,2.5)--(.5,5);
\draw[thick,rounded corners,color=blue] (0,1.5)--(1.5,1.5)--(1.5,2.5)--(2.5,2.5)--(2.5,3.5)--(3.5,3.5)--(3.5,4.5)--(4.5,4.5)--(4.5,5);
\draw[thick,rounded corners,color=blue] (0,.5)--(.5,.5)--(.5,2.5)--(1.5,2.5)--(1.5,3.5)--(2.5,3.5)--(2.5,5);
\end{tikzpicture}\quad
\begin{tikzpicture}[x=1.5em,y=1.5em]
\draw[step=1,gray, thin] (0,0) grid (5,5);
\draw[color=black, thick](0,0)rectangle(5,5);
\draw[thick,rounded corners,color=blue](0,4.5)--(3.5,4.5)--(3.5,5);
\draw[thick,rounded corners,color=blue] (0,3.5)--(1.5,3.5)--(1.5,5);
\draw[thick,rounded corners,color=blue] (0,2.5)--(.5,2.5)--(.5,5);
\draw[thick,rounded corners,color=blue] (0,1.5)--(.5,1.5)--(.5,2.5)--(2.5,2.5)--(2.5,3.5)--(3.5,3.5)--(3.5,4.5)--(4.5,4.5)--(4.5,5);
\draw[thick,rounded corners,color=blue] (0,.5)--(.5,.5)--(.5,1.5)--(1.5,1.5)--(1.5,3.5)--(2.5,3.5)--(2.5,5);
\end{tikzpicture}\quad
\begin{tikzpicture}[x=1.5em,y=1.5em]
\draw[step=1,gray, thin] (0,0) grid (5,5);
\draw[color=black, thick](0,0)rectangle(5,5);
\draw[thick,rounded corners,color=blue](0,4.5)--(3.5,4.5)--(3.5,5);
\draw[thick,rounded corners,color=blue] (0,3.5)--(1.5,3.5)--(1.5,5);
\draw[thick,rounded corners,color=blue] (0,2.5)--(.5,2.5)--(.5,5);
\draw[thick,rounded corners,color=blue] (0,1.5)--(.5,1.5)--(.5,2.5)--(1.5,2.5)--(1.5,3.5)--(3.5,3.5)--(3.5,4.5)--(4.5,4.5)--(4.5,5);
\draw[thick,rounded corners,color=blue] (0,.5)--(.5,.5)--(.5,1.5)--(1.5,1.5)--(1.5,2.5)--(2.5,2.5)--(2.5,5);
\end{tikzpicture}
\]
Pictured below are the non-reduced pipe dreams for $w$.
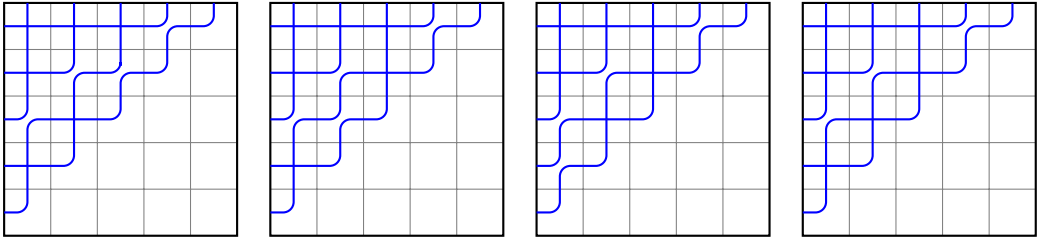
\[
\begin{tikzpicture}[x=1.5em,y=1.5em]
\draw[step=1,gray, thin] (0,0) grid (5,5);
\draw[color=black, thick](0,0)rectangle(5,5);
\draw[thick,rounded corners,color=blue](0,4.5)--(3.5,4.5)--(3.5,5);
\draw[thick,rounded corners,color=blue] (0,3.5)--(1.5,3.5)--(1.5,5);
\draw[thick,rounded corners,color=blue] (0,2.5)--(.5,2.5)--(.5,5);
\draw[thick,rounded corners,color=blue] (0,1.5)--(1.5,1.5)--(1.5,3.5)--(2.5,3.5)--(2.5,3.5)--(2.5,5);
\draw[thick,rounded corners,color=blue] (0,.5)--(.5,.5)--(.5,2.5)--(2.5,2.5)--(2.5,3.5)--(3.5,3.5)--(3.5,4.5)--(4.5,4.5)--(4.5,5);
\end{tikzpicture}\quad
\begin{tikzpicture}[x=1.5em,y=1.5em]
\draw[step=1,gray, thin] (0,0) grid (5,5);
\draw[color=black, thick](0,0)rectangle(5,5);
\draw[thick,rounded corners,color=blue](0,4.5)--(3.5,4.5)--(3.5,5);
\draw[thick,rounded corners,color=blue] (0,3.5)--(1.5,3.5)--(1.5,5);
\draw[thick,rounded corners,color=blue] (0,2.5)--(.5,2.5)--(.5,5);
\draw[thick,rounded corners,color=blue] (0,1.5)--(1.5,1.5)--(1.5,2.5)--(2.5,2.5)--(2.5,3.5)--(2.5,5);
\draw[thick,rounded corners,color=blue] (0,.5)--(.5,.5)--(.5,2.5)--(1.5,2.5)--(1.5,3.5)--(3.5,3.5)--(3.5,4.5)--(4.5,4.5)--(4.5,5);
\end{tikzpicture}\quad
\begin{tikzpicture}[x=1.5em,y=1.5em]
\draw[step=1,gray, thin] (0,0) grid (5,5);
\draw[color=black, thick](0,0)rectangle(5,5);
\draw[thick,rounded corners,color=blue](0,4.5)--(3.5,4.5)--(3.5,5);
\draw[thick,rounded corners,color=blue] (0,3.5)--(1.5,3.5)--(1.5,5);
\draw[thick,rounded corners,color=blue] (0,2.5)--(.5,2.5)--(.5,5);
\draw[thick,rounded corners,color=blue] (0,1.5)--(.5,1.5)--(.5,2.5)--(2.5,2.5)--(2.5,3.5)--(2.5,5);
\draw[thick,rounded corners,color=blue] (0,.5)--(.5,.5)--(.5,1.5)--(1.5,1.5)--(1.5,3.5)--(3.5,3.5)--(3.5,4.5)--(4.5,4.5)--(4.5,5);
\end{tikzpicture}\quad
\begin{tikzpicture}[x=1.5em,y=1.5em]
\draw[step=1,gray, thin] (0,0) grid (5,5);
\draw[color=black, thick](0,0)rectangle(5,5);
\draw[thick,rounded corners,color=blue](0,4.5)--(3.5,4.5)--(3.5,5);
\draw[thick,rounded corners,color=blue] (0,3.5)--(1.5,3.5)--(1.5,5);
\draw[thick,rounded corners,color=blue] (0,2.5)--(.5,2.5)--(.5,5);
\draw[thick,rounded corners,color=blue] (0,1.5)--(1.5,1.5)--(1.5,3.5)--(2.5,3.5)--(3.5,3.5)--(3.5,4.5)--(4.5,4.5)--(4.5,5);
\draw[thick,rounded corners,color=blue] (0,.5)--(.5,.5)--(.5,2.5)--(2.5,2.5)--(2.5,3.5)--(2.5,5);
\end{tikzpicture}
\]
\end{example}

Theorem~\ref{thm:pipedreams} has the following corollary, which makes it clear that there is no ambiguity in talking about the ``highest degree part" of $\fG_w(\bx; \by)$. 

\begin{cor}\label{cor:high_degree_xy}
Let $w$ be a permutation and let $d$ be the degree of $\fG_w(\bx)$. Then $\fG_w(\bx; \by)$ has terms which are of bidegree $(d,d)$ in the $\bx$ and $\by$ variables, and no term in $\fG_w(\bx; \by)$ has $\bx$-degree or $\by$-degree higher than $d$. 
\end{cor}

\begin{proof}
Let $P$ be a pipe dream for $w$ with $|P|=c$. We see that $P$ contributes a monomial of degree $c$ to $\fG_w(\bx)$, and that all monomials of the same degree occur with the same sign, so there is no cancellation. 
Thus, the degree of $\fG_w(\bx)$ is the maximum number of crosses in any pipe dream for $w$.

Then, from the formula for $\fG_w(\bx; \by)$, we see that a pipe dream with $c$ crosses contributes terms in bidegrees $(a,b)$ for $a$, $b \leq c$. So no term in $\fG_w(\bx; \by)$ can have $\bx$-degree or $\by$-degree larger than $d$.
Moreover, those pipe dreams with the maximal number of crosses all contribute monomials in bidegree $(d,d)$ with the same sign, so there can be no cancellation, and we see that $\fG_w(\bx; \by)$ has terms of bidegree $(d,d)$. 
\end{proof}

Define the  \newword{Castlenuovo-Mumford polynomial} $\fCM(\bx)$ to be $(-1)^{\deg \fG_w(\bx) - \inv(w)}$ times the highest degree part of $\fG_w(\bx)$ and define the \newword{double Castlenuovo-Mumford polynomial}, $\fCM_w(\bx; \by)$, to be $(-1)^{\inv(w)}$ times the highest degree part of $\fG_w(\bx; \by)$. 
These sign factors make $\fCM_w(\bx; \by)$ and $\fCM_w(\bx)$ have positive coefficients, as can be seen from Theorem~\ref{thm:pipedreams}.
So the degree of $\fCM_w(\bx)$, and the bidegree of $\fCM_w(\bx; \by)$, are both given by the maximal number of crosses in any pipe dream for $w$.

If we specialize the single Grothendieck polynomial $\fG_w(\bx)$ by setting $x_i \mapsto 1-t$, we obtain the K-polynomial of the matrix Schubert variety $X_w$ \cite{Knutson.Miller}. Note that this specialization does not affect the degrees of the polynomials, since all top-degree terms of $\fG_w(\bx)$ have the same sign. Moreover, $\deg \fS_w(\bx) = \codim X_w = \inv(w)$ \cite{Fulton}. Thus, we can rewrite Lemma~\ref{lem:Benedetti} as follows.

\begin{corollary}\label{cor:reg}
	Let $X_w$ be a matrix Schubert variety. Then ${\rm reg} \; X_w = \deg \fG_w(\bx) - \inv(w)$.
\end{corollary} 

\section{Simple properties of Rajchgot index}
\label{sec:rajchgot}

In this section, we develop some of the basic combinatorial properties of the Rajchgot index of a permutation, which we can prove without introducing new tools.
In the next section, we will build more complex tools and prove stronger results.

\subsection{Rajchgot index and the inversion statistic}
\begin{lemma}\label{lem:rajinv_inequality}
	Let $w$ be a permutation and let $(\ell_1(w), \ell_2(w), \ldots, \ell_n(w))$ be its inversion code. Then $r_i(w) \geq \ell_i(w)$. Hence, $\raj(w) \geq \inv(w)$.
\end{lemma}
\begin{proof}
	To make an increasing sequence starting at $w_i$, we must at least delete all following letters that are less than $w_i$. There are $\ell_i(w)$ such letters.
\end{proof}

In light of Lemma~\ref{lem:rajinv_inequality}, it is natural to ask when we have equality.

\begin{proposition}\label{prop:raj=inv}
	Let $w$ be a permutation. Then $r_i(w) = \ell_i(w)$ if $w$ has no $132$ pattern starting at position $i$. In particular, $\raj(w) = \inv(w)$ if and only if $w$ avoids the pattern $132$, i.e.\ if $w$ is dominant.
\end{proposition}
\begin{proof}
	If $r_i(w) > \ell_i(w)$, then there exist $j,k$ such that $i<j<k$ and $w_i < w_k < w_j$, which is exactly a $132$ pattern. Conversely, if there is a $132$ pattern starting at position $i$, then $r_i(w) > \ell_i(w)$.
\end{proof}

The next proposition summarizes standard results about Schubert and Grothendieck polynomials of dominant permutations.   
\begin{proposition}\label{prop:dominant_background}
Let $w$ be a permutation in $S_n$.
Then $w$ is dominant if and only if $\ell_1(w) \geq \ell_2(w) \geq \cdots \geq \ell_n(w)$. In this case, we can view $(\ell_1(w), \ldots, \ell_n(w))$ as a partition, and the Rothe diagram $\Rothe(w)$ is the Young diagram of that partition.
This gives a bijection between dominant permutations in $S_n$ and partitions fitting inside the staircase $(n-1, n-2, \ldots, 2,1)$.

If $w$ is a dominant permutation, then 
\[ \begin{array}{rcl}
\fG_w(\bx; \by)&=& \displaystyle\prod_{(i,j)\in \Rothe(w) }(x_i+y_j-x_iy_j)\\
\fS_w(\bx; \by)&=&\displaystyle \prod_{(i,j)\in \Rothe(w) }(x_i-y_j)\\
\fG_w(\bx) = \fS_w(\bx) &=&\displaystyle  \prod_{(i,j)\in \Rothe(w) }x_i= \prod_{i=1}^n x_i^{\ell_i(w)} . \\
\end{array}
\]
\end{proposition}
\begin{proof}
	That $w$ is dominant if and only if $\ell_1(w) \geq \ell_2(w) \geq \cdots \geq \ell_n(w)$ follows from \cite[Section~2.2.1]{Manivel}.
	When $w$ is dominant, it has a single pipe dream whose crossing tiles are exactly the elements of $\Rothe(w)$.
\end{proof}

\subsection{Rajchgot index and major index}

Let $w \in S_n$ be a permutation and let $m_i(w)$ equal the number of $j \geq i$ such that $w_j > w_{j+1}$.  The \newword{major index} of $w$ is \[
\maj(w) = \sum_{i=1}^n m_i(w).
\]

We remind the reader that Theorem~\ref{thm:raj=mm} will establish a formula for $\raj$ as a maximum of many values of $\maj$. At the moment, we can prove an inequality:

\begin{lemma}\label{lem:rajmaj_inequality}
	Let $w$ be a permutation and let $m_i(w)$ equal the number of $j \geq i$ such that $w_j > w_{j+1}$. 
	Then $r_i(w) \geq m_i(w)$. Hence,  $\raj(w) \geq \maj(w)$.
\end{lemma}
\begin{proof}
	To make an ascending sequence from $w_i w_{i+1} \dots w_n$, we must delete from each consecutive descending sequence at least all but one letter. Hence $r_i(w) \geq m_i(w)$ for all $i$. Since, $\maj(w) = \sum_{i=1}^n m_i(w)$, we then have $\raj(w) \geq \maj(w)$.
\end{proof}

\subsection{Fireworks permutations}
We previously explained that dominant permutations have equality between Rajchgot index and the inversion statistic; we now discuss when we have equality between Rajchgot index and major index.
 In order to state this characterization, we must define a new class of permutations.

\begin{definition}
	The permutation $w$ is \newword{fireworks} if the initial elements of its decreasing runs are in increasing order. (These permutations are sometimes called $3-12$ avoiding.)  Write $\mathcal F_n$ for the set of fireworks permutations in $S_n$.
We define $w$ to be \newword{inverse fireworks} if $w^{-1}$ is fireworks.
\end{definition}

\begin{example}
	The permutation $41|62|853|97$ is fireworks because $4 < 6 < 8 < 9$.
\end{example}

\begin{remark}
We imagine each descending run in a permutation as a firework dropping to earth. In a fireworks show, each new explosive is launched higher than the previous one.
\end{remark}

\begin{proposition}\label{prop:raj=maj}
	Let $w$ be a permutation. Then $\raj(w) = \maj(w)$ if and only if $w$ is fireworks.
\end{proposition}
\begin{proof}
	If $w$ is fireworks, we can delete all but the first letter of each descending run to get an ascending sequence, so $r_i(w) \leq m_i(w)$. However, we already know the reverse inequality by Lemma~\ref{lem:rajmaj_inequality}. Thus $r_i(w) = m_i(w)$ for all $i$ and so $\raj(w) = \maj(w)$.
	
	Now suppose $w$ is not fireworks. Then we have a $3-12$ pattern, i.e.\ there are $i<j$ such that $w_j < w_{j+1} < w_i$. Now consider $w_i w_{i+1} \dots w_n$. It does not suffice to delete $m_i(w)$ letters because we must delete both $w_j$ and $w_{j+1}$, so $m_i(w) < r_i(w)$, and so $\maj(w) < \raj(w)$.
\end{proof}

A \newword{set partition} of $[n]$ is a collection $\pi$ of pairwise-disjoint nonempty subsets of $[n]$ whose union is $[n]$; the subsets are called the \newword{blocks} of $\pi$. We canonically order the blocks of $\pi$ according to their largest elements and index them as $(\pi_t, \pi_{t+1}, \ldots, \pi_n)$, where $\max(\pi_t) < \max(\pi_{t+1}) < \cdots < \max(\pi_n)$. 
(This apparently odd decision to number ending at $n$ rather than starting at $1$ will pay off in simpler notation later on; see Remark~\ref{rem:WhyEndAtN}.) 

The number of set partitions of $[n]$ is called the $n$th \newword{Bell number} \cite[A000110]{OEIS}. For brevity, we generally omit commas and braces from the individual parts of a set partition, for example, writing $( 1,\ 3,\ 45,\ 67,\ 8,\ 29 )$ rather than $\{ \{ 1 \},\ \{ 3 \},\ \{ 4,\ 5 \},\ \{ 6,\ 7 \},\ \{ 8 \},\ \{ 2,9 \} \}$.

\begin{proposition}[\cite{Claesson}]\label{prop:bell}
	Fireworks permutations are enumerated by the Bell numbers.
\end{proposition}
\begin{proof}
	The decreasing runs give blocks of a set partition of $[n]$. Conversely, if we order the blocks of a set partition $\pi$ according to their largest elements, and write each block backwards, we get a fireworks permutation.
\end{proof}

\begin{example}
	The fireworks permutation $416285397$ corresponds under this bijection to the set partition $(14,26,358,79)$.
\end{example}

This is a convenient time to prove the following lemma, which we will want later.

\begin{lemma} \label{lem:fireworks_delete_position}
Let $w$ be a permutation in $S_n$ for $n \geq 2$ and let $w'$ be the permutation in $S_{n-1}$ such that $w'(1)$, $w'(2)$, \dots, $w'(n-1)$ are linearly ordered in the same way as $w(2)$, $w(3)$, \dots, $w(n)$. If $w$ is fireworks, then $w'$ is fireworks and, if $w$ is inverse fireworks, then $w'$ is inverse fireworks.
\end{lemma}

\begin{proof}
We first consider the case that $w$ is fireworks. The descending runs of $w'$ are the same as those of $w$, with the initial letter missing.

We now consider the case that $w$ is inverse fireworks, so we must consider descending runs of $w^{-1}$. 
If $w(1) = 1$, then $1$ is in a descending run of $w^{-1}$ by itself, and deleting it will not change the initial elements of the other descending runs.
If $w(1) > 1$, then $1$ is not the initial element of its descending run in $w^{-1}$ and deleting it does not change the initial elements of the descending runs.
\end{proof}

\subsection{Valley and inverse valley permutations} \label{sec:valley}

A permutation $w$ is called a \newword{valley permutation} if there is some index $a$ for which $w(1) > w(2) > \cdots > w(a) < w(a+1) < \cdots < w(n)$. 
There are $2^{n-1}$ valley permutations in $S_n$, because a valley permutation is uniquely determined by the subset $\{ w(1), w(2),\ldots, w(a-1) \}$ of $\{ 2,3,\ldots, n \}$. 
We define a permutation $w$ to be a \newword{inverse valley permutation} if $w^{-1}$ is a valley permutation.

\begin{lemma} \label{lem:dominant_and_fireworks_is_inverse_valley}
A permutation $w$ is a valley permutation if and only if $w$ is both dominant and inverse fireworks; $w$ is inverse valley if and only if $w$ is both dominant and fireworks.
\end{lemma}
\begin{proof}
We note that the valley permutations are the permutations which simultaneously avoid $132$ and $231$; and that the inverse valley permutations are those which simultaneously avoid $132$ and $312$. By symmetry, it is sufficient to show that the inverse valley permutations are exactly the dominant permutations that are fireworks.

The easy direction is that a inverse valley permutation is both dominant and fireworks. Indeed, let $w$ be a inverse valley permutation. Since $w$ avoids $132$, it is dominant. 
Since $w$ avoids $312$, it avoids $3-12$, and therefore it is fireworks. 

Now suppose that $w^{-1}$ is not a valley permutation. Then, there is some $p$ with $w^{-1}(p-1) <w^{-1}(p)$ and $w^{-1}(p+1) <w^{-1}(p)$. If $w^{-1}(p-1) < w^{-1}(p+1) < w^{-1}(p)$, then $w^{-1}$, and hence $w$, is not dominant, so we may assume that $w^{-1}(p+1) < w^{-1}(p-1) < w^{-1}(p)$. Put $i=w^{-1}(p+1)$, $j = w^{-1}(p-1)$ and $k = w^{-1}(p)$, so that $i<j<k$. We now consider two cases. If $w(k-1) < w(k)=p$, then we have $w(k-1) < w(k) < w(j)$, showing that $w$ is not fireworks. If $w(k-1)>w(k)=p$, then $w(i) < w(k) < w(k-1)$, showing that $w$ is not dominant. Either way, the lemma follows.
\end{proof}

\begin{remark}
There is a standard bijection between dominant permutations in $S_n$, and partitions $\lambda$ contained the staircase $(n-1, n-2, \ldots, 1,0)$: The Rothe diagram of a dominant permutation is such a partition. 
Under this bijection, valley permutations correspond to partitions with distinct parts, and inverse valley permutations correspond to partitions whose transposes have distinct parts.
\end{remark}

\section{The shape of a permutation and associated ideas} \label{sec:fireworks}

\subsection{The blob diagram} We begin by giving a pictorial description of the Rajchgot code. 
Start by drawing the graph of $w$, i.e. by plotting dots in the $n\times n$ grid in positions $(i, w_i)$ for $i\in [n]$.  
(To be clear, we are using matrix coordinates, so $(i,j)$ is in row $i$, numbered from the top, and is in column $j$, numbered from the left.)
Draw a lasso around the dots which are maximally southeast in the grid.  Call this set of dots $B_n(w)$. Then draw another lasso around the maximally southeast non-lassoed dots.  This next set of dots is $B_{n-1}(w)$.  Continue in this way until all dots have been lassoed. We refer to this picture as the \newword{blob diagram} of $w$ (see Figure~\ref{fig:Davids_picture} for an example).

\begin{figure}[ht]
\[
\begin{tikzpicture}[x=2em,y=2em]
\draw[step=1,gray, thin] (0,0) grid (9,9);
\draw[color=black, thick](0,0)rectangle(9,9);
\node at (3.5,8.5) {$\bullet$};
\node at (5.5,7.5) {$\bullet$};
\node at (1.5,6.5) {$\bullet$};
\node at (2.5,5.5) {$\bullet$};
\node at (4.5,4.5) {$\bullet$};
\node at (6.5,3.5) {$\bullet$};
\node at (8.5,2.5) {$\bullet$};
\node at (0.5,1.5) {$\bullet$};
\node at (7.5,0.5) {$\bullet$};
\node at (-0.5,8.5) {$1$};
\node at (-0.5,7.5) {$2$};
\node at (-0.5,6.5) {$3$};
\node at (-0.5,5.5) {$4$};
\node at (-0.5,4.5) {$5$};
\node at (-0.5,3.5) {$6$};
\node at (-0.5,2.5) {$7$};
\node at (-0.5,1.5) {$8$};
\node at (-0.5,0.5) {$9$};
\node at (0.5,9.5) {$1$};
\node at (1.5,9.5) {$2$};
\node at (2.5,9.5) {$3$};
\node at (3.5,9.5) {$4$};
\node at (4.5,9.5) {$5$};
\node at (5.5,9.5) {$6$};
\node at (6.5,9.5) {$7$};
\node at (7.5,9.5) {$8$};
\node at (8.5,9.5) {$9$};
\node at (8.0,1.5) {$B_9$};
\node at (3.5,2.5) {$B_8$};
\node at (5.0,6.0) {$B_7$};
\node at (3.0,7.0) {$B_6$};
\node at (1.0,6.5) {$B_5$};
\draw[blue,thick, rotate around={-30:(1,6.5)}] (1,6.5) ellipse (.75 and .75);
\draw[blue,thick, rotate around={-20:(3,7)}] (3,7) ellipse (.75 and 2);
\draw[blue,thick, rotate around={-20:(5,6)}] (5,6) ellipse (.75 and 2);
\draw[blue,thick, rotate around={-70:(3.5,2.5)}] (3.5,2.5) ellipse (.75 and 3.5);
\draw[blue,thick, rotate around={-30:(8,1.5)}] (8,1.5) ellipse (.7 and 1.5);
\end{tikzpicture}
\]	
\caption{For $w=462357918$, we plot $w$ in the $9 \times 9$ grid as shown. The blobs are found by repeatedly lassoing the maximally southeast unlassoed $\bullet$s. 
} \label{fig:Davids_picture}
\end{figure}
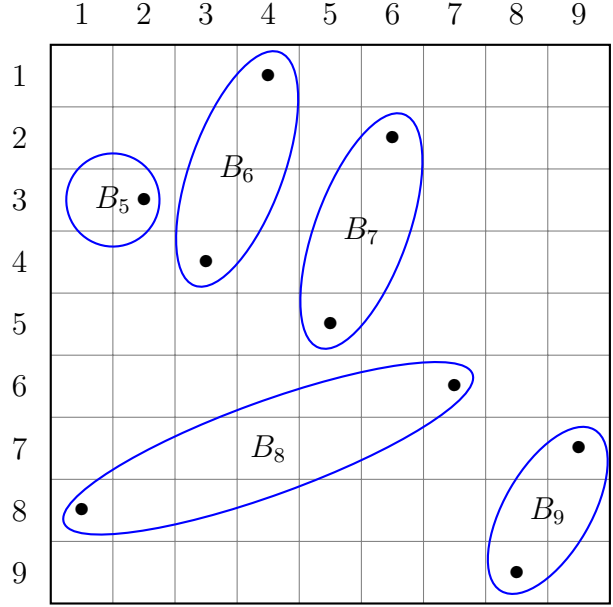

\begin{lemma}\label{lem:blob_LIS}
Let $(i,w_i)$ be in blob $B_k$. Then the longest increasing subsequence starting at $(i,w_i)$ contains $n+1-k$ elements.
\end{lemma}
\begin{proof}
We prove this claim by reverse induction on $k$. If $(i,w_i)$ is in $B_n$, then there are no dots to the southeast of $(i, w_i)$, so the longest increasing subsequence starting at $(i, w_i)$ is just the singleton $\{ w_i \}$. 

Now, suppose that $(i, w_i)$ is in $B_k$. Let a longest increasing subsequence starting at $w_i$ continue $w_i w_j \cdots$. Then $w_j$ is in $B_{\ell}$ for $\ell>k$, so the length of the sequence starting at $w_j$ is at most $n-\ell+1 \leq n-(k+1)+1=n-k$. So the length of the sequence starting at $w_i$ is at most $n-k$. This shows that any  increasing subsequence starting at $w_i$ has length at most $n-k$.

Conversely, since $(i,w_i)$ is in blob $B_k$, there is some element $(j, w_j)$ in blob $B_{k+1}$ to the southeast of $(i,w_i)$, and by induction there is an  increasing subsequence of length $n-k-1$ starting at $w_j$. Prepending $w_i$ to this subsequence gives an  increasing subsequence of length $n-k$ starting at $w_i$.
\end{proof}

\begin{cor} \label{cor:raj_code_from_blobs}
Let $(r_1(w), r_2(w), \ldots, r_n(w))$ be the Rajchgot code of $w$, and let $(i, w_i)$ be in blob $B_k$. Then $r_i(w) = k - i$.
\end{cor}
\begin{proof}
By Lemma~\ref{lem:blob_LIS}, the longest increasing subsequence starting at $w_i$ has $n-k+1$ elements. There are $n-i+1$ elements in the word $w_i w_{i+1} \cdots w_n$, so $k-i$ of them are omitted. 
\end{proof}

\subsection{The set partition and shape of a permutation} 
First, let $\epsilon(w) = \epsilon_1 \epsilon_2 \cdots \epsilon_n$ be the word where $(i, w_i)$ is in blob $B_{\epsilon_i}$. In other words, we project the blob numbers left onto the rows of the blob diagram. 
So, for the permutation in Figure~\ref{fig:Davids_picture}, we have $\epsilon(w) = 675678989$. 
From Corollary~\ref{cor:raj_code_from_blobs}, the Rajchgot code of $w$ is $\epsilon(w) - (1,2,\ldots, n)$ so, in this case, the Rajchgot code of $462357918$ is $(6,7,5,6,7,8,9,8,9) - (1,2,3,4,5,6,7,8,9) = (5,5,2,2,2,2,2)$.
Note that, if we projected the blob numbers onto the columns instead, we would obtain $\epsilon(w^{-1})$. 
We note that the words $\epsilon(w)$ and $\epsilon(w^{-1})$ are anagrams of each other. In the example of Figure~\ref{fig:Davids_picture}, $\epsilon(w^{-1})$ is $856677899$.

Define $\pi_k(w)$ to be the set of column labels of the dots in $B_k(w)$.  (Note that by symmetry, $\pi_k(w^{-1})$ is the row labels of the dots in $B_k(w)$.) Then $\pi(w)$ is the set partition of $w$. Note that $i$ and $j$ are in the same block of the set partition exactly if $(w^{-1}(i), i)$ and $(w^{-1}(j), j)$ are in the same blob.
In our running example, $\pi(w)$ is the set partition $\{ \{ 2 \},\ \{ 3,4 \},\ \{ 5,6 \},\ \{ 1,7 \},\ \{ 8,9 \} \}$. We will generally shorten this to $\{ 2,\ 34,\ 56,\ 17,\ 89 \}$. 
(Of course, we would obtain $\pi(w^{-1})$ by recording the row labels of the entries in each blob.)
We note that the ordering of the blocks of $\pi(w)$ is recoverable from the set partition $\pi(w)$, because the maximum elements of the blocks occur in increasing order.
We index the blocks of $\pi(w)$ as $(\pi_t, \pi_{t+1}, \ldots, \pi_n)$, so that $\pi_k$ corresponds to block $B_k$, and we set $\alpha_k(w) = \# \pi_k(w)$. 
We define the composition $(\alpha_t(w), \alpha_{t+1}(w), \ldots, \alpha_n(w))$ to be the \newword{shape} of $w$. 
We note that $\alpha_k$ is the number of times the letter $k$ occurs in the word $\epsilon(w^{-1})$. 
In particular, since $\epsilon(w)$ and $\epsilon(w^{-1})$ are anagrams, we see that $w$ and $w^{-1}$ have the same shape. 
We can express the Rajchgot index in terms of the shape:

\begin{lemma}\label{lem:set_partition_to_raj}
For $w\in S_n$, we have 
\[ \raj(w)=\sum_{k=1}^n k\alpha_k - \binom{n+1}{2} = \sum_{k=1}^n (\alpha_k + \alpha_{k+1} + \cdots + \alpha_n) - \binom{n+1}{2} . \] 
\end{lemma}

\begin{remark} \label{rem:WhyEndAtN}
Lemma~\ref{lem:set_partition_to_raj} is the first of several places where our indexing convention simplifies formulas.
\end{remark}

\begin{proof}[Proof of Lemma~\ref{lem:set_partition_to_raj}]
Let $\epsilon(w) = (\epsilon_1, \epsilon_2, \ldots, \epsilon_n)$. From Corollary~\ref{cor:raj_code_from_blobs}, we have $\raj(w) = \sum_{i=1}^n \epsilon_i - \binom{n+1}{2}$. 
By definition, $\alpha_k$ is the number of times that $k$ occurs in this sum. This proves the first formula for $\raj(w)$; the second is a formal rearrangement of it.
\end{proof}

\begin{cor}
We have $\raj(w) = \raj(w^{-1})$. 
\end{cor}
\begin{proof}
	We have seen that $w$ and $w^{-1}$ have the same shape. By Lemma~\ref{lem:set_partition_to_raj}, Rajchgot index only depends on shape.
\end{proof}

Given two compositions, $(\alpha_j, \alpha_{j+1}, \ldots, \alpha_n)$ and $(\beta_k, \beta_{k+1}, \ldots, \beta_n)$ of $n$, we say that $\alpha$ \newword{dominates} $\beta$ and write $\alpha \succeq \beta$ if $\alpha_m + \alpha_{m+1} + \cdots + \alpha_n \geq \beta_m + \beta_{m+1} + \cdots + \beta_n$ for all $m$. We write $\alpha \succ \beta$ to mean $\alpha \succeq \beta$ and $\alpha \neq \beta$.

\begin{cor} \label{cor:dominates_raj}
Let $u$ and $v$ be permutations of shapes $\alpha$ and $\beta$. If $\alpha \succeq \beta$, then $\raj(u) \geq \raj(v)$; if $\alpha \succ \beta$, then $\raj(u) > \raj(v)$.
\end{cor}

\begin{proof}
This is clear from the second formula for $\raj$ in Lemma~\ref{lem:set_partition_to_raj}.
\end{proof}

We return to our discussion of valley permutations, from Section~\ref{sec:valley}. 
\begin{lemma} \label{lem:valley_shape}
For each composition $\alpha$ of $n$, there is exactly one valley permutation, $f_{\alpha}$ of shape $\alpha$, and likewise one inverse valley permutation of shape $\alpha$, which is $f_{\alpha}^{-1}$.
\end{lemma}
\begin{proof}
Inverting permutations preserves the shape while switching valley and inverse valley, so it is enough to prove the valley case.

Suppose $\alpha = (\alpha_t, \dots, \alpha_n)$.
For $t\leq k \leq n$, define $\rho_{k} = (n+1) - \alpha_n - \alpha_{n-1} - \cdots - \alpha_k$. 
Let $R = \{ \rho_{t} < \rho_{t+1} < \cdots < \rho_n \}$.
Let $[n] \setminus R = \{ \lambda_1 > \lambda_2 > \cdots > \lambda_{t-1}  \}$. 
Then the corresponding valley permutation $f$ has $f(k) = \lambda_k$ for $1 \leq k \leq t-1$ and $f(k) = \rho_k$ for $t \leq k \leq n$. 

\begin{figure}[ht]
\[
\begin{tikzpicture}[x=2em,y=2em]
\draw[step=1,gray, thin] (0,0) grid (9,9);
\draw[color=black, thick](0,0)rectangle(9,9);
\node at (7.5,8.5) {$\bullet$};
\node at (6.5,7.5) {$\bullet$};
\node at (4.5,6.5) {$\bullet$};
\node at (1.5,5.5) {$\bullet$};
\node at (0.5,4.5) {$\bullet$};
\node at (2.5,3.5) {$\bullet$};
\node at (3.5,2.5) {$\bullet$};
\node at (5.5,1.5) {$\bullet$};
\node at (8.5,0.5) {$\bullet$};
\node at (-0.5,8.5) {$1$};
\node at (-0.5,7.5) {$2$};
\node at (-0.5,6.5) {$3$};
\node at (-0.5,5.5) {$4$};
\node at (-0.5,4.5) {$5$};
\node at (-0.5,3.5) {$6$};
\node at (-0.5,2.5) {$7$};
\node at (-0.5,1.5) {$8$};
\node at (-0.5,0.5) {$9$};
\node at (0.5,9.5) {$1$};
\node at (1.5,9.5) {$2$};
\node at (2.5,9.5) {$3$};
\node at (3.5,9.5) {$4$};
\node at (4.5,9.5) {$5$};
\node at (5.5,9.5) {$6$};
\node at (6.5,9.5) {$7$};
\node at (7.5,9.5) {$8$};
\node at (8.5,9.5) {$9$};
\node at (8.65,0.9) {$B_9$};
\node at (6.5,4.5) {$B_8$};
\node at (4.0,4.0) {$B_7$};
\node at (2.5,3.85) {$B_6$};
\node at (1.0,5.0) {$B_5$};
\draw[blue,thick, rotate around={-45:(1,5)}] (1,5) ellipse (.5 and .95);
\draw[blue,thick, rotate around={-40:(2.5,3.8)}] (2.5,3.8) ellipse (.5 and 0.7);
\draw[blue,thick, rotate around={-15:(4,4.4)}] (4,4.4) ellipse (.48 and 2.8);
\draw[blue,thick, rotate around={-15:(6.4,4.77)}] (6.4,4.7) ellipse (1.2 and 4.5);
\draw[blue,thick, rotate around={-30:(8.5,0.5)}] (8.5,0.5) ellipse (.5 and 0.9);
\end{tikzpicture}
\]	
\caption{For $\alpha = (2,1,2,3,1)$, the proof of Lemma~\ref{lem:valley_shape} gives $R = \{1 < 3 < 4 < 6 < 9 \}$ and $[9] \backslash R = \{ 8 > 7 > 5 > 2 \}$. So the corresponding valley permutation is $f = 875213469$. Here we have drawn the blob diagram of $f$, so the reader may observe that $f$ has shape $\alpha$, as desired.}
\label{fig:valley_figure}
\end{figure}
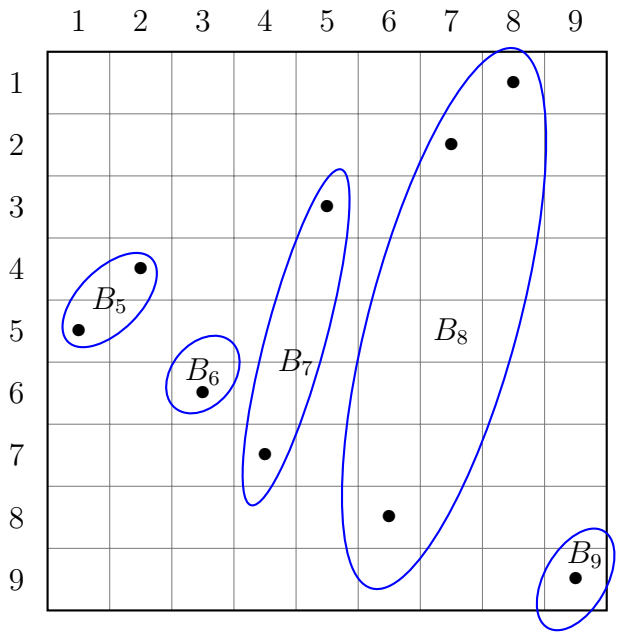

From the blob diagram for $f$ (for example, see Figure~\ref{fig:valley_figure}), it is easy to see that there are $n-t+1$ blobs, with $B_k$ (for $t \leq k \leq n$) containing the dots in columns $\{ \rho_k, \rho_k +1, \rho_k +2, \ldots, \rho_{k+1}-1 \}$. Thus, the shape of $f$ is $\alpha$. Uniqueness follows from the bijection between valley permutations and subsets of $\{ 2,3,\ldots, n \}$.
\end{proof}

\subsection{Going down edges in weak order}

Suppose that we have covers $w >_R w s_i$ or $w >_L s_i w$. In this section, we will discuss how the blobs, the Rajchgot codes, and the Rajchgot index change in this case.

\begin{lemma}\label{lem:cover_blob}
Let $w >_R w s_i$. Let $(i, w(i))$ and $(i+1, w(i+1))$ be in blobs $B_p$ and $B_q$ of $w$, respectively. If $p>q$, then $(i+1, w(i))$ and $(i, w(i+1))$ are, respectively, in blobs $B_p$ and $B_q$ of $w s_i$, and all other dots stay in the same blobs. In this case, $w$ and $w s_i$ have the same shape, and $\raj(w) = \raj(w s_i)$. If $p \leq q$, then $w$ and $w s_i$ have different shapes. Calling their shapes $\alpha$ and $\beta$ respectively, we have $\alpha \succ \beta$ and $\raj(w) > \raj(w s_i)$.
\end{lemma}

\begin{example}\label{ex:going_down_edges}
	Let $w = 462357918$ as in Figure~\ref{fig:Davids_picture}. Let $w' = ws_2$. The blob diagram of $w'$ is
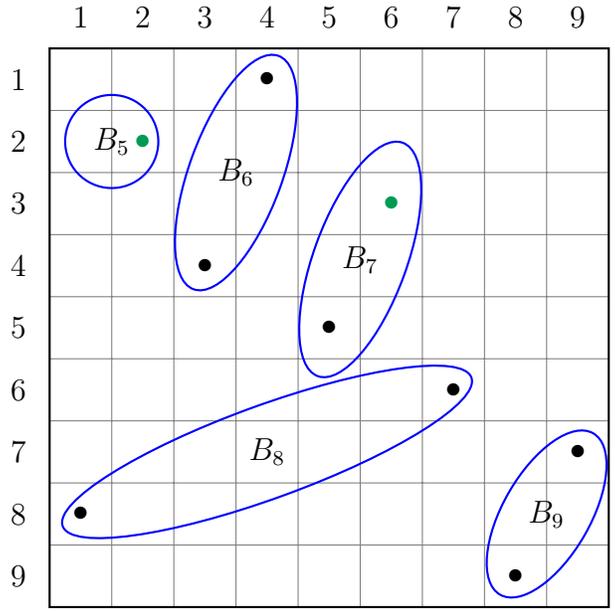
\[	
\begin{tikzpicture}[x=2em,y=2em]
\draw[step=1,gray, thin] (0,0) grid (9,9);
\draw[color=black, thick](0,0)rectangle(9,9);
\node at (3.5,8.5) {$\bullet$};
\node at (5.5,6.5) {${\textcolor{ForestGreen} \bullet}$};
\node at (1.5,7.5) {${\textcolor{ForestGreen} \bullet}$};
\node at (2.5,5.5) {$\bullet$};
\node at (4.5,4.5) {$\bullet$};
\node at (6.5,3.5) {$\bullet$};
\node at (8.5,2.5) {$\bullet$};
\node at (0.5,1.5) {$\bullet$};
\node at (7.5,0.5) {$\bullet$};
\node at (-0.5,8.5) {$1$};
\node at (-0.5,7.5) {$2$};
\node at (-0.5,6.5) {$3$};
\node at (-0.5,5.5) {$4$};
\node at (-0.5,4.5) {$5$};
\node at (-0.5,3.5) {$6$};
\node at (-0.5,2.5) {$7$};
\node at (-0.5,1.5) {$8$};
\node at (-0.5,0.5) {$9$};
\node at (0.5,9.5) {$1$};
\node at (1.5,9.5) {$2$};
\node at (2.5,9.5) {$3$};
\node at (3.5,9.5) {$4$};
\node at (4.5,9.5) {$5$};
\node at (5.5,9.5) {$6$};
\node at (6.5,9.5) {$7$};
\node at (7.5,9.5) {$8$};
\node at (8.5,9.5) {$9$};
\node at (8.0,1.5) {$B_9$};
\node at (3.5,2.5) {$B_8$};
\node at (5.0,5.6) {$B_7$};
\node at (3.0,7.0) {$B_6$};
\node at (1.0,7.5) {$B_5$};
\draw[blue,thick, rotate around={-30:(1,7.5)}] (1,7.5) ellipse (.75 and .75);
\draw[blue,thick, rotate around={-20:(3,7)}] (3,7) ellipse (.75 and 2);
\draw[blue,thick, rotate around={-20:(5,5.6)}] (5,5.6) ellipse (.75 and 2);
\draw[blue,thick, rotate around={-70:(3.5,2.5)}] (3.5,2.5) ellipse (.75 and 3.5);
\draw[blue,thick, rotate around={-30:(8,1.5)}] (8,1.5) ellipse (.7 and 1.5);
\end{tikzpicture},
\] where we have marked the dots that have moved with respect to $w$ in green. Observe that the shape of the permutation is unchanged from that of $w$ (see Figure~\ref{fig:Davids_picture}).
	
	Then let $w'' = ws_4$. Note that $w <_R w''$. Then the blob diagram for $w''$ is
	
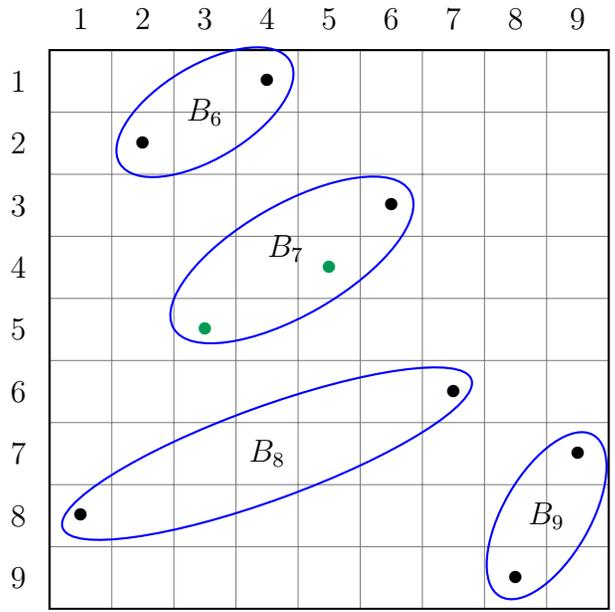
\[	
\begin{tikzpicture}[x=2em,y=2em]
\draw[step=1,gray, thin] (0,0) grid (9,9);
\draw[color=black, thick](0,0)rectangle(9,9);
\node at (3.5,8.5) {$\bullet$};
\node at (1.5,7.5) {$\bullet$};
\node at (5.5,6.5) {$\bullet$};
\node at (4.5,5.5) {${\textcolor{ForestGreen} \bullet}$};
\node at (2.5,4.5) {${\textcolor{ForestGreen} \bullet}$};
\node at (6.5,3.5) {$\bullet$};
\node at (8.5,2.5) {$\bullet$};
\node at (0.5,1.5) {$\bullet$};
\node at (7.5,0.5) {$\bullet$};
\node at (-0.5,8.5) {$1$};
\node at (-0.5,7.5) {$2$};
\node at (-0.5,6.5) {$3$};
\node at (-0.5,5.5) {$4$};
\node at (-0.5,4.5) {$5$};
\node at (-0.5,3.5) {$6$};
\node at (-0.5,2.5) {$7$};
\node at (-0.5,1.5) {$8$};
\node at (-0.5,0.5) {$9$};
\node at (0.5,9.5) {$1$};
\node at (1.5,9.5) {$2$};
\node at (2.5,9.5) {$3$};
\node at (3.5,9.5) {$4$};
\node at (4.5,9.5) {$5$};
\node at (5.5,9.5) {$6$};
\node at (6.5,9.5) {$7$};
\node at (7.5,9.5) {$8$};
\node at (8.5,9.5) {$9$};
\node at (8.0,1.5) {$B_9$};
\node at (3.5,2.5) {$B_8$};
\node at (3.8,5.8) {$B_7$};
\node at (2.5,8.0) {$B_6$};
\draw[blue,thick, rotate around={-59:(2.5,8)}] (2.5,8) ellipse (.75 and 1.6);
\draw[blue,thick, rotate around={-60:(3.9,5.62)}] (3.9,5.62) ellipse (.89 and 2.2);
\draw[blue,thick, rotate around={-70:(3.5,2.5)}] (3.5,2.5) ellipse (.75 and 3.5);
\draw[blue,thick, rotate around={-30:(8,1.5)}] (8,1.5) ellipse (.7 and 1.5);
\end{tikzpicture},
\] where again we have marked the dots that have moved in green. Note that here the shape is different from the shape of $w$, as depicted in Figure~\ref{fig:Davids_picture}. Letting $\alpha=(2,3,2,2)$ be the shape of $w''$ and $\beta = (1,2,2,2,2)$ be the shape of $w$, we observe that $\alpha \succ \beta$.
\end{example}

\begin{proof}[Proof of Lemma~\ref{lem:cover_blob}]
First, suppose that $p>q$. Throughout the lassoing process up to blob $B_{p+1}$, exactly the same dots will get lassoed in $w$ and in $w s_i$. When we get to blob $B_p$, because $j>k$, there is some dot to the southeast of $(i+1, w(i+1))$ and this dot will remain to the southeast of $(i, w(i+1))$, so $(i, w(i+1))$ will still not be lassoed; on the other hand, $(i+1, w(i))$ will still be lassoed. All other dots behave exactly the same way for $w$ and $w s_i$. After blob $B_p$ has been formed, the remaining unlassoed dots for $w$ and $w s_i$ are precisely the same except for whether they have an empty row in row $i$ or in row $i+1$, so everything is the same after that point. Since we have shown that the blobs of $w$ and $w s_i$ have the same number of dots, $w$ and $w s_i$ have the same shape and $\raj(w) = \raj(w s_i)$.

Now, suppose that $p\leq q$. Again, the lassoing process is the same up to blob $B_{q+1}$. In $w s_i$, the dot $(i+1, w(i))$ is southeast of $(i, w(i+1))$, so $(i, w(i+1))$ cannot be added to blob $B_q$. If $p=q$, then $(i+1, w(i))$ does get added to blob $B_q$; if $p<q$, then whatever dot was southeast of $(i, w(i))$ to prevent it from getting added to blob $B_q$ of $w$ also prevents $(i+1, w(i))$ from getting added to blob $B_q$ of $w$. So either way,  blob $B_q$ for $w s_i$ has one fewer element than for $w$. As we proceed through the lassoing process, each dot occurs in a blob of $w s_i$ whose number is less than or equal to the corresponding blob of $w$. This shows that $\alpha \succeq \beta$ and, since blob $B_q$ is smaller for $w s_i$ than for $w$, we have $\alpha \succ \beta$. By Corollary~\ref{cor:dominates_raj}, this implies that $\raj(w) > \raj(w s_i)$.
\end{proof}

We have a similar result for left covers, whose proof is analogous and thus omitted:
\begin{lemma}\label{lem:cover_blob_left}
Let $w >_L s_j w$. Let $(w^{-1}(j), j)$ and $(w^{-1}(j+1), j+1)$ be in blobs $B_p$ and $B_q$ of $w$, respectively. 
If $p>q$, then $(w^{-1}(j), j+1)$ and $(w^{-1}(j+1), j)$ are, respectively, in blobs $B_p$ and $B_q$ of $s_j w$, and all other dots stay in the same blobs. In this case, $w$ and $s_j w$ have the same shape, and $\raj(w) = \raj(s_j w)$.
 If $p \leq q$, then $w$ and $s_j w$ have different shapes. Calling their shapes $\alpha$ and $\beta$, respectively, we have $\alpha \succ \beta$ and $\raj(w) > \raj(s_j w)$. \qed 
\end{lemma}

Lemmas~\ref{lem:cover_blob} and~\ref{lem:cover_blob_left} immediately give the following corollary.

\begin{cor} \label{cor:LR_raj_ineq}
Let $u$ and $v$ be permutations with $u \geq_{LR} v$. Then $\raj(u) \geq \raj(v)$. We have $\raj(u) = \raj(v)$ if and only if $u$ and $v$ have the same shape. \qed
\end{cor}

Finally, we focus on understanding how the Rajchgot code of $w$ is related to that of $w s_i$ or $s_i w$, in the case where these permutations have the same Rajchgot index.

\begin{lemma} \label{lem:cover_rajcode}
If $w >_L s_i w$ and $\raj(w) = \raj(s_i w)$, then $\rajc(w) = \rajc(s_i w)$. 
Suppose that $w >_R w s_i$ and $\raj(w) = \raj(w s_i)$. Let $\rajc(w) = (r_1, r_2, \ldots, r_n)$. Then $\rajc(w s_i) = (r_1, r_2, \ldots, r_{i+1}+1, r_i-1, \ldots, r_n)$. 
\end{lemma}

\begin{proof}
We begin by analyzing the case that $w >_L w s_i$. In this case, when we compare the blob diagrams of $w$ and $s_i w$, dots which are in the same row are in the same blob. 
Thus, $\epsilon(w) = \epsilon(s_i w)$ and so $\rajc(w) = \rajc(s_i w)$.

We now consider the case that $w >_R w s_i$. In this case, the dots which are in rows other than $i$ and $i+1$ are in the same blob for $w$ and for $w s_i$, but the dots in those two rows switch blobs. So, if $\epsilon(w) = (\epsilon_1, \epsilon_2, \ldots, \epsilon_n)$, then $\epsilon(w s_i) = (\epsilon_1, \epsilon_2, \ldots, \epsilon_{i+1}, \epsilon_i, \ldots, \epsilon_n)$. We now apply the formula $\rajc(w) = \epsilon(w) - (1,2,\ldots, n)$ from Corollary~\ref{cor:raj_code_from_blobs}.
\end{proof}

\subsection{The fireworks map} We now describe how we can use the blob diagram to see whether a permutation is fireworks.
\begin{lemma}\label{lem:consecutiveblobs}
	The permutation $w\in S_n$ is fireworks if and only if the dots in each blob occupy consecutive rows of the graph of $w$.  Likewise, $w$ is inverse fireworks if and only if the dots in each blob occupy consecutive columns.
\end{lemma}

\begin{proof}
Suppose $w$ is fireworks. Let the maximal descending runs of $w$ be $D^t$, \dots, $D^{n-1}$, $D^n$ and write each such run $D^k$ as $D^k_1$, $D^k_2$, \dots, $D^k_j$. 

 We claim that dots of each $B_k$ are those coming from the run $D^k$. To see this, first note that dots from the last descending run $D^n$ are a subset of $B_n$. By the fireworks condition, each dot from $D^{n-1}$ is northwest of the dot for $D^n_1$. Therefore none of these dots can be in $B_n$; hence, they are all in $B_{n-1}$. Similarly each dot from $D^{n-2}$ is northwest of the dot for $D^{n-1}_1$, so they cannot be in $B_{n-1}$ and hence must be in $B_{n-2}$, etc.
 
 Since the descending runs of $w$ are by definition in consecutive positions, it follows that the dots of each $B_k$ occupy consecutive rows of the graph.

Conversely, suppose that the dots of each $B_k$ are in consecutive positions. Clearly, $w$ restricted to the rows of any $B_k$ is descending.
We will show that the blobs are exactly the descending runs.
Let blob $B_k$ be in rows $p$, $p+1$, \dots, $q-1$ and let blob $B_{k+1}$ be in rows $q$, $q+1$, \dots, $r-1$. We must show that $w_{q-1} < w_q$. 
Indeed, if not, then we have $w_p > w_{p+1} > \cdots > w_{q-1} > w_q > w_{q+1} > \cdots > w_{r-1}$. But then all the dots of $B_k$ are to the northeast of all the dots of $B_{k+1}$, contradicting that each dot of $B_k$ must be northwest of at least one dot in $B_{k+1}$. 
So, we now know that the blobs are precisely the descending runs of $w$, and we know that the rightmost elements of the blobs come in increasing order, so $w$ is fireworks.

 The statement about inverse fireworks permutations now follows by transposition.
\end{proof}

We now describe a map, the \newword{fireworks map}, which turns an arbitrary permutation $w$ into a fireworks permutation $\Phi(w)$. 
The fireworks permutation $\Phi(w)$ corresponds to the set partition $\pi(w)$ using the bijection of Proposition~\ref{prop:bell}.  
In other words, we take the dots in the graph of $w$ and shove the dots of each blob into consecutive rows. 

For example, let $w=462357918$. We computed before that $\pi(w) = \{ 2,\ 34,\ 56,\ 17,\ 89 \}$. The corresponding fireworks permutation is $243657198$. 
See Figure~\ref{fig:FireworksMap} for a graphical depiction of this process.

\begin{figure}[ht]
\[
\begin{tikzpicture}[x=2em,y=2em]
\draw[step=1,gray, thin] (0,0) grid (9,9);
\draw[color=black, thick](0,0)rectangle(9,9);
\node at (3.5,7.5) {$\bullet$};
\node at (5.5,5.5) {$\bullet$};
\node at (1.5,8.5) {$\bullet$};
\node at (2.5,6.5) {$\bullet$};
\node at (4.5,4.5) {$\bullet$};
\node at (6.5,3.5) {$\bullet$};
\node at (8.5,1.5) {$\bullet$};
\node at (0.5,2.5) {$\bullet$};
\node at (7.5,0.5) {$\bullet$};
\node at (-0.5,8.5) {$1$};
\node at (-0.5,7.5) {$2$};
\node at (-0.5,6.5) {$3$};
\node at (-0.5,5.5) {$4$};
\node at (-0.5,4.5) {$5$};
\node at (-0.5,3.5) {$6$};
\node at (-0.5,2.5) {$7$};
\node at (-0.5,1.5) {$8$};
\node at (-0.5,0.5) {$9$};
\node at (0.5,9.5) {$1$};
\node at (1.5,9.5) {$2$};
\node at (2.5,9.5) {$3$};
\node at (3.5,9.5) {$4$};
\node at (4.5,9.5) {$5$};
\node at (5.5,9.5) {$6$};
\node at (6.5,9.5) {$7$};
\node at (7.5,9.5) {$8$};
\node at (8.5,9.5) {$9$};
\draw[blue,thick, rotate around={-60:(1.5,8.5)}] (1.5,8.5) ellipse (.5 and 1);
\draw[blue,thick, rotate around={-45:(3,7)}] (3,7) ellipse (.5 and 1.7);
\draw[blue,thick, rotate around={-45:(5,5)}] (5,5) ellipse (.5 and 1.7);
\draw[blue,thick, rotate around={-80:(3.5,3)}] (3.5,3) ellipse (.5 and 4);
\draw[blue,thick, rotate around={-45:(8,1)}] (8,1) ellipse (.5 and 1.7);
\end{tikzpicture}
\]	
\caption{The fireworks permutation $\Phi(w)$ corresponding to the permutation $w$ from Figure~\ref{fig:Davids_picture}. Note that the the dots in blob $B_k$ are in the same columns in both blob diagrams.
}
\label{fig:FireworksMap}
\end{figure}
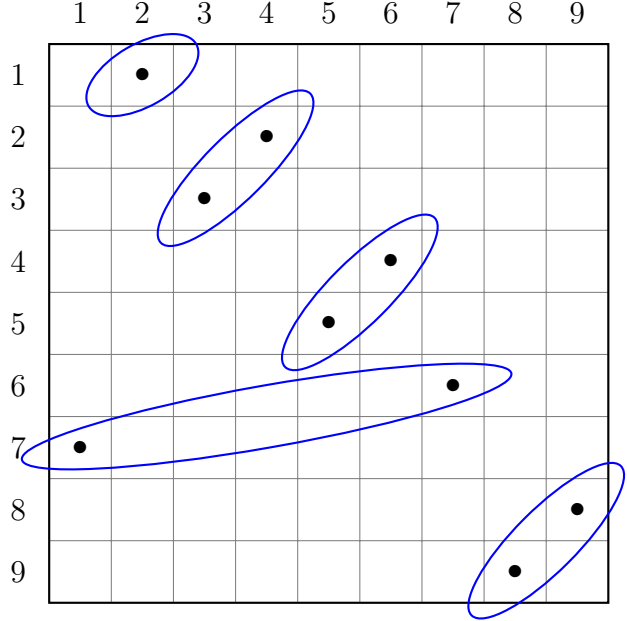

We define $\Phi_{\rm inv}(w) = \Phi(w^{-1})^{-1}$. Graphically, we take the blobs of $w$ and shove them into consecutive columns.

\begin{lemma}\label{lem:fireworks_down}
	For any permutation $w$, we have $\Phi(w) \leq_R w$ and $\Phi_{\rm inv}(w)\leq_L w$. In other words, we have length additive factorizations $w = \Phi(w) v = u \Phi_{\rm inv}(w)$ for some $u$ and $v$.
\end{lemma}
\begin{proof}
	Suppose $i<j$ and $i$ appears left of $j$ in $w$. Then $i$ goes in a smaller numbered block of $\pi(w)$, so $i$ is left of $j$ in $\Phi(w)$. Thus $\Phi(w) \leq_R w$ by \cite[Lemma~4.1]{Hammett.Pittel} (see also, \cite{Berge}).

Now notice $\Phi(w^{-1})\leq_R w^{-1}$. Thus $\Phi_{\rm inv}(w)=\Phi(w^{-1})^{-1}\leq_L w$.
\end{proof}

\begin{remark}
	Although $\Phi(w) \leq_R w$, this does not mean that $\Phi$ is order-preserving! For example, $\Phi(4312) = 1432$ and $\Phi(3412) = 3142$. Now $4312 = 3412 \cdot s_1$, so $3412 <_R 4312$, but $3142 \not \leq_R 1432$ since $1$ and $3$ are noninverted in $1432$, but inverted in $3142$.
\end{remark}

\begin{lemma}
\label{lemma:stayfireworks}
The blobs $B_k$ of $\Phi(w)$ and of $w$ consist of dots in the same columns, the permutation $w$ is fireworks if and only if $\Phi(w) = w$, and the permutations $w$ and $\Phi(w)$ have the same shape. 
The corresponding statements (replacing ``columns'' with ``rows'') also hold for $\Phi_{\rm inv}$ and being inverse fireworks. 
\end{lemma}
\begin{proof}
The first claim is true by construction.
The second claim is straightforward from Lemma~\ref{lem:consecutiveblobs}.
As we checked in the proof of Lemma~\ref{lem:consecutiveblobs}, the blobs of a fireworks permutation are its descending runs. 
The cardinalities of these blobs are the shape, proving the last claim for $\Phi$. The claims for $\Phi_{\rm inv}$ are then immediate by transposition.
\end{proof}

\begin{corollary}\label{cor:fireworksraj}
For any permutation $w$, we have $\raj(w) = \raj(w^{-1})= \raj(\Phi(w)) = \raj(\Phi_{\rm inv}(w))$.
\end{corollary}

\begin{proof}
Since $w$, $w^{-1}$, $\Phi(w)$ and $\Phi_{\rm inv}(w)$ all have the same shape by Lemma~\ref{lemma:stayfireworks}, this corollary follows immediately from Lemma~\ref{lem:set_partition_to_raj}.
\end{proof}

In the case of the inverse fireworks map, we can state a stronger result.

\begin{cor}\label{cor::leftfireworks_perserves_code}
	For any permutation $w$, we have $\rajc(w) = \rajc(\Phi_{\rm inv}(w))$.
\end{cor}

\begin{proof}
The blobs $B_k$ of the blob diagrams for $w$ to $\Phi_{\rm inv}(w)$ consist of dots from the same rows. Therefore, $\epsilon(w) = \epsilon(\Phi_{\rm inv}(w))$. We obtain the Rajchgot code by subtracting $(1,2,3,\ldots,n)$ from the word $\epsilon$.
\end{proof}

For a composition $\alpha = (\alpha_1, \dots, \alpha_r)$ of $n$, let $S_{\alpha}$ be the Young subgroup $S_{\alpha_1} \times  S_{\alpha_2} \times \cdots \times S_{\alpha_r}$ (in the standard way) and let $e_\alpha$ be its longest element. 
A permutation $w$ is called \newword{layered} if $w = e_\alpha$ for some $\alpha$.
(These are also called ``$231$- and $312$-avoiding" permutations.)

\begin{lemma} \label{lem:phi_phiL_layered}
Let $w$ be a permutation of shape $\alpha$. Then $\Phi(\Phi_{\rm inv}(w)) = \Phi_{\rm inv}(\Phi(w)) = e_{\alpha}$.
\end{lemma}

\begin{proof}
This is clear from the blob diagram description and Lemma~\ref{lemma:stayfireworks}.
\end{proof}

We are now ready to prove:
\begin{theorem} \label{thm:raj=mm}
For all $w \in S_n$, we have 
 	\[ 	\raj(w) = \max \{\maj(v) : v \leq_R w \} = \max \{\maj(u^{-1}) : u \leq_L w \} \]
\end{theorem}

\begin{proof}
Fix $w \in S_n$.
From Corollary~\ref{cor:LR_raj_ineq}, for any $v$ with $v \leq_R w$, we have $\raj(w) \geq \raj(v)$. And, from Lemma~\ref{lem:rajmaj_inequality}, we have $\raj(v) \geq \maj(v)$. So $\raj(w) \geq  \max \{\maj(v) : v \leq_R w \}$. We now must show that there is some $v$ with $w \geq_R v$ and $\raj(w) = \maj(v)$.

Indeed, we have $w \geq_R \Phi(w)$ (Lemma~\ref{lem:fireworks_down}), we have $\raj(w) = \raj(\Phi(w))$ (Corollary~\ref{cor:fireworksraj}) and we have $\raj(\Phi(w)) = \maj(\Phi(w))$ (Proposition~\ref{prop:raj=maj}), so $v=\Phi(w)$ is the desired permutation. 
The proof that $\raj(w) = \max \{\maj(u^{-1}) : u \leq_L w \}$ is similar.
\end{proof}

\subsection{Weak order and factorizations}

We now discuss interactions between the fireworks maps and the left and right weak orders.

\begin{lemma}\label{lem:fireworks_factorizations}
	Suppose $w \in S_n$ has shape $\alpha$. Then $w$ has a unique length-additive factorization
	\[
	w = ue_\alpha v,
	\]
	for some permutations $u$, $v \in S_n$. Moreover, we have $\Phi(w) = ue_\alpha$ and $\Phi_{\rm inv}(w) = e_\alpha v$. 
\end{lemma}

The reader may wish to consult Example~\ref{eg:posetExample} now.

\begin{proof}
We first show that such a factorization exists. By Lemma~\ref{lem:fireworks_down}, we have length additive factorizations $w = \Phi(w) v = (u \Phi_{\rm inv}(\Phi(w))) v = u e_{\alpha} v$, where we have used Lemma~\ref{lem:phi_phiL_layered} in the last equality. For this $u$ and $v$, we have $\Phi(w) = w v^{-1} = u e_{\alpha}$. Similarly, there is some {\it a priori} different $u'$ and $v'$ for which we have length-additive factorizations $w = u' \Phi_{\rm inv}(w) = u' (\Phi(\Phi_{\rm inv}(w)) v') = u' e_{\alpha} v'$ and, for this $(u', v')$, we have $\Phi_{\rm inv}(w) = e_{\alpha} v'$. 

We will now show that a permutation $w$ of shape $\alpha$ can have only one length-additive factorization of the  form $u e_{\alpha} v$.
This will prove the uniqueness claim in the lemma, and will also establish that $(u,v) = (u', v')$, so that we have both $\Phi(w) = ue_\alpha$ and $\Phi_{\rm inv}(w) = e_\alpha v$.

First, suppose that we have a length-additive factorization $w = u e_{\beta} v$ but do not yet assume that $\beta$ is the shape of $w$ (which we still denote $\alpha$.)
Let $A_1, A_2, \dots, A_n$ be the blobs of $w$ and let $B_1, B_2, \dots, B_n$ be the blobs of $e_\beta$ (some of these blobs with small indices may be empty). Let $(i,j)$ be in blob $A_k$. 

\begin{claim}
	We have $(v(i), u^{-1}(j)) \in B_1 \cup B_2 \cup \dots \cup B_k$. 
\end{claim}
\begin{proof}[Proof of Claim]
Let $(i,j)=(i_1, j_1)$ and consider a maximal chain \[
(i_1, j_1), (i_2, j_2), \dots, (i_{n-k+1}, j_{n-k+1})
\]
 of dots arranged northwest to southeast in the graph of $w$ (i.e., such that $i_h < i_{h+1}$ and $j_h < j_{h+1}$ for all $h$). By \cite[Lemma~4.1]{Hammett.Pittel} and the fact that $u e_\beta v$ is a length-additive factorization, the dots \[
 (v(i_1), u^{-1}(j_1)), (v(i_2), u^{-1}(j_2)), \dots, (v(i_{n-k+1}), u^{-1}(j_{n-k+1}))
 \]
  are also arranged northwest to southeast in the graph of $e_\beta$ (although no longer necessarily a maximal such chain). Therefore, $(v(i), u^{-1}(j))$ is in $B_1 \cup B_2 \cup \cdots \cup B_k$.
\end{proof}

From the claim, we deduce that $|A_1 \cup A_2 \cup \dots \cup A_k| \leq |B_1 \cup B_2 \cup \dots \cup B_k|$ or, in other words, that $\alpha_1+\alpha_2+\cdots+\alpha_k \leq \beta_1 + \beta_2 + \dots + \beta_k$. 

Now, suppose $\alpha = \beta$. Then we have equalities in the previous paragraph, implying that we must have $(v(i), u^{-1}(j))$ in $B_k$ for every $(i,j)$ in $A_k$. So $u$ takes the columns of blob $B_k$ to the columns of blob $A_k$, and $v$ takes the rows of blob $B_k$ to the rows of blob $A_k$.
Moreover, since any two rows within a blob are inverted in $e_\alpha$, and $u e_\alpha$ and $e_\alpha v$ are length-additive factorizations, the permutations $u$ and $v$ must preserve the order within blobs (again by \cite[Lemma~4.1]{Hammett.Pittel}). Thus, $u$ and $v$ are uniquely determined by $e_\alpha$. 
\end{proof}

We can use this factorization to understand the $\leq_{LR}$ interval $[e_{\alpha}, w]_{LR}$.

\begin{lem} \label{lem:interval_structure}
Let $w$ be a permutation of shape $\alpha$, and let $w = u e_{\alpha} v$ be the unique length-additive factorization of $w$. Then the map $\mu: (u', v') \mapsto u' e_{\alpha} v'$ is a poset isomorphism from the product $[e_\alpha, u]_L \times [e_\alpha,v]_R$ to the two-sided weak interval $[e_{\alpha}, w]_{LR}$.
\end{lem}
\begin{proof}
The map $\mu$ lands in $[e_{\alpha}, w]_{LR}$ and is injective by the uniqueness part of Lemma~\ref{lem:fireworks_factorizations}, combined with Lemmas~\ref{lem:cover_blob} and~\ref{lem:cover_blob_left}. It is also a morphism of posets. 
We will prove that $\mu$ is surjective and that is is a poset isomorphism at the same time. Specifically, we will show the following claim.
\begin{claim}
	 If $w_1 \lessdot_{LR} w_2$ is a cover, with $w_1 \in [e_{\alpha}, u e_{\alpha} v]_{LR}$ and $w_2 \in \im \mu$, then $w_1 \in \im \mu$, and $\mu^{-1}(w_2)$ covers $\mu^{-1}(w_1)$ in $[e, u]_L \times [e,v]_R$.
\end{claim}
\begin{proof}[Proof of Claim]
Let $w_1 \lessdot_{LR} w_2$. Then either $w_1 s_i = w_2$ or $s_i w_1 = w_2$; without loss of generality, we consider only the former case $w_1 s_i = w_2$. 
Since we assume $w_2 \in \im \mu$, we have a length-additive factorization $w_2 = u_2 e_{\alpha} v_2$. We claim that $v_2 >_R v_2 s_i$. Indeed, from Lemma~\ref{lem:cover_blob}, $(i, w_2(i))$ is in a higher-numbered blob of $w_2$ than $(i+1, w_2(i+1))$ is. From the description in the proof of Lemma~\ref{lem:fireworks_factorizations} of how to recover $v_2$ from $w_2$, this means that $v_2(i) > v_2(i+1)$. Thus, $v_2 >_R v_2 s_i$. So $w_1 = u_2 e_{\alpha} (v_2 s_i)$ is also a length-additive factorization, and the cover $(w_1, w_2)$ lifts to the cover $((u_2, v_2 s_i),\ (u_2, v_2))$. This proves the claim.
\end{proof}

 Let us see why the claim establishes the surjectivity and isomorphism. For surjectivity, let $x \in [e_{\alpha}, w]_{LR}$ and choose a saturated $\leq_{LR}$-chain from $x$ up to $w$. Working down that chain from the top, the claim inductively shows that every element of that chain is in the image of $\mu$ (the base case $x=w$ is trivial), so in particular $x \in \im \mu$; since $x$ was arbitrary, this shows that $\mu$ is surjective. Also, since the claim shows that every edge of the Hasse diagram of  $[e_{\alpha}, w]_{LR}$ lifts to an edge of the Hasse diagram of $[e, u]_L \times [e,v]_R$, and we have already noted that $\mu$ is a map of posets, the claim also shows that $\mu$ is an isomorphism of posets.
\end{proof}

We can use weak order to get a new perspective on the set of fireworks permutations of a given shape.
\begin{lem} \label{lem:maximal_is_valley}
Let $f_{\alpha}$ be the unique valley permutation of shape $\alpha$ (introduced in Lemma~\ref{lem:valley_shape}). Then the set of fireworks permutations of shape $\alpha$ is the left interval $[e_{\alpha}, f_{\alpha}^{-1}]_L$, and the set of inverse fireworks permutations of shape $\alpha$ is the right interval $[e_{\alpha}, f_{\alpha}]_R$. 
\end{lem}

\begin{proof}
We will show that the set of inverse fireworks permutations of shape $\alpha$ is  $[e_{\alpha}, f_{\alpha}]_R$; the other claim follows by inversion.
Let $\alpha = (\alpha_t, \alpha_{t+1},\ldots, \alpha_n)$. Set $\rho_{n+1} = n+1$ and $\rho_k = n+1-\alpha_k-\alpha_{k+1} - \dots - \alpha_n$ for $t \leq k \leq n$. 

In any inverse fireworks permutation $w$ of shape $\alpha$, Lemma~\ref{lem:consecutiveblobs} shows that blob $B_k$ is in columns $\{ \rho_k, \rho_k +1 ,\ldots, \rho_n \}$, and the rows of that blob are $\pi_k(w)$ by definition. 
Let us consider whether or not $(i,j)$ will be an inversion of $w$, for some $1 \leq i< j \leq n$. 
Recall that $(i,j)$ is an inversion if and only if the dot in column $i$ in the blob diagram is east of the dot in column $j$.
This configuration occurs if and only if $\rho_k \leq i < j < \rho_{k+1}$. 

Suppose, then, that $\rho_k \leq i < \rho_{k+1} \leq \rho_{\ell} \leq j < \rho_{\ell+1}$.
If $w$ is $e_{\alpha}$, $(i,j)$ will not be an inversion. 
If $w$ is $f_{\alpha}$, $(i,j)$ will not be an inversion if $j=\rho_{\ell}$, but will be if $j > \rho_{\ell}$. For an arbitrary inverse fireworks permutation $w$, it is also true that $(i,j)$ is not an inversion if $j = \rho_{\ell}$, as the largest element of $\pi_{\ell}$ must be larger than all elements of $\pi_k$. Hence, the inversion set of $w$ contains that of $e_\alpha$ and is contained in that of $f_\alpha$, implying $w \in [e_{\alpha}, f_{\alpha}]_R$ by \cite[Lemma~4.1]{Hammett.Pittel}. 

We have now shown that the set of all inverse fireworks permutations of shape $\alpha$ is contained in $[e_{\alpha}, f_{\alpha}]_R$. Conversely, let $u \in [e_{\alpha}, f_{\alpha}]_R$, and apply Lemma~\ref{lem:interval_structure} with $w = f_{\alpha}$. Then we see that $u$ has a length-additive factorization $e_{\alpha} v$, and that $u$ is of shape $\alpha$. These conditions force $u$ to be inverse fireworks as well.
\end{proof}

\begin{rem}
It is well known that the interval $[e_{\alpha}, w_0]_R$ has \[\frac{n!}{(\alpha_t)! (\alpha_{t+1})! \cdots (\alpha_n)!}\] elements. 
One can show that the subinterval $[e_{\alpha}, f_{\alpha}]_R$ has \[
\frac{n!}{(\alpha_t)! (\alpha_{t+1})! \cdots (\alpha_n)! } \frac{ \prod_{k=t}^n \alpha_k}{\prod_{k=t}^n (\alpha_k+\alpha_{k+1} + \cdots + \alpha_n)}
\] elements. 
\end{rem}

We are now able to prove the following lemma, to be used in Section~\ref{sec:rajchgotpolynomials}.

\begin{lemma} \label{lem:ealpha_lemma}
Let $x$ be any permutation in $S_n$, let $\alpha = (\alpha_k, \alpha_{k+1}, \ldots, \alpha_n)$ be a composition of $n$, and let $y = e_{\alpha} \ast x \ast e_{\alpha}$, where $\ast$ denotes the Demazure product. Then $\raj(y) \geq \raj(e_{\alpha})$ and we have equality if and only if $y = e_{\alpha}$.
\end{lemma}
\begin{proof}
Set $\beta_m = \alpha_k + \alpha_{k+1} + \cdots + \alpha_m$ and let $\rho_m = \{ \beta_{m-1}+1, \beta_{m-1}+2, \cdots, \beta_{m-1}+\alpha_m \}$, so $\rho$ is a set partition of $[n]$ with $|\rho_m| = \alpha_m$.
Write $\pi$ for the set partition $\pi(y)$ associated to $y$.

We note that $e_{\alpha} \ast e_{\alpha} = e_{\alpha}$. Therefore, we have $e_{\alpha} \ast y = y = y \ast e_{\alpha}$. 
The equality $y = y \ast e_{\alpha}$ is equivalent to imposing that $y$ is descending when restricted to any of the blocks of $\rho$. In other words, if we split the blob diagram of $y$ into horizontal strips of sizes $\alpha_k$, $\alpha_{k+1}$, \dots, $\alpha_n$, the dots in each strip are arranged on a northeast-southwest antidiagonal. We thus see that $\pi_n \cup \pi_{n-1} \cup \cdots \cup \pi_m \supseteq \rho_n \cup \rho_{n-1} \cup \cdots \cup \rho_m$. 
Together with Lemma~\ref{lem:set_partition_to_raj}, this implies that $\raj(y) \geq \raj(e_{\alpha})$, with equality if and only if $\pi_m = \rho_m$ for all $m$.

However, we also have $y = e_{\alpha} \ast y$, which means that if we split the blob diagram of $y$ into vertical strips of the same sizes, the dots in each strip are also arranged on a northeast-southwest antidiagonal. We thereby deduce that we have equality if and only if $y^{-1}(\pi_m) = \rho_m$. Thus, we have equality if and only if $y$ maps each $\rho_m$ to itself, and does so in an order reversing way. The only permutation which does this is $e_{\alpha}$.
\end{proof}

\begin{example} \label{eg:posetExample} 
We list all permutations of shape $\alpha = (1,1,2)$, their factorizations in the form $u e_{\alpha} v$, and the corresponding double Castelnuovo--Mumford polynomials:
\[ 
\begin{array}{|ccc|}
\hline
2341 &          & \\
1342 &          & \\
1243 & 1423 & 4123 \\
\hline
\end{array} \qquad
\begin{array}{|rrr|}
\hline
s_1 s_2 e_{112} && \\
s_2 e_{112} && \\
e_{112} & e_{112} s_2 & e_{112} s_2 s_1 \\
\hline
\end{array}\]
\[ \begin{array}{|lll|}
\hline
(x_1 x_2 x_3) (y_1^3) && \\
(x_1 x_2 x_3) (y_1^2 y_2 + y_1 y_2^2) && \\
(x_1 x_2 x_3) (y_1 y_2 y_3) & (x_1^2 x_2 + x_1 x_2^2) (y_1 y_2 y_3) & (x_1^3) (y_1 y_2 y_3) \\
\hline
\end{array}\]
The layered permutation $e_{112}=1243$ is in the lower left, the valley permutation $f_{112} = 4123$ is in the lower right and the inverse valley permutation $f_{112}^{-1} = 2341$ is in the upper left. 
The permutations in the left column are fireworks; the permutations in the bottom row are inverse fireworks, and the permutations which are maximally northeast are dominant.
The maps $\Phi$ and $\Phi_{\rm inv}$ are the orthogonal projections onto the left column and bottom row, respectively.
\end{example}

\section{Degrees of Grothendieck polynomials}
\label{sec:degrees}

In this section, we prove part of Theorem~\ref{thm:degree}. Namely, we establish that the degree of the Castelnuovo--Mumford polynomial $\fCM_w(\bx)$ is the Rajchgot index $\raj(w)$, and hence that $\raj(w) - \inv(w)$ gives the Castelnuovo--Mumford regularity of the matrix Schubert variety $X_w$. This also  proves the final remaining equality of Theorem~\ref{thm:main_mm} (the other equalities of Theorem~\ref{thm:main_mm} were established in Theorem~\ref{thm:raj=mm}). The remainder of Theorem~\ref{thm:degree}, the identification of leading monomials of Castelnuovo--Mumford polynomials, will be proved later after we establish the factorization statement of Theorem~\ref{thm:main}.

\begin{lemma}\label{lem:divided_diffs}
	If $u \leq_{LR} w$, then $\deg \fCM_u(\bx) \leq \deg \fCM_w(\bx)$.
\end{lemma}
\begin{proof}
Suppose $w > ws_i$. Then $\fG_{ws_i}(\bx) = \overline{\partial}_i \fG_{w}(\bx) = \partial_i(\fG_w(\bx) - x_{i+1} \fG_w(\bx))$. Suppose $\deg \fG_w(\bx) = d$. Then $\deg( \fG_w(\bx) - x_{i+1} \fG_w(\bx)) = d+1$. Note that $\partial_i$ is a linear operator that maps any $k$-form to a $(k-1)$-form (or to $0$). Hence, $\deg \fG_{ws_i}(\bx) \leq d = \deg \fG_w(\bx)$.

Suppose $v > s_iv$. Let $w = v^{-1}$, so $(s_iv)^{-1} = ws_i$. Then $w > ws_i$, so we have $\deg \fG_{ws_i}(\bx) \leq \deg \fG_w(\bx)$. But by the pipe dream formula (Theorem~\ref{thm:pipedreams}), it is clear that $\deg \fG_w(\bx) = \deg \fG_v(\bx)$ and $\deg \fG_{ws_i}(\bx) = \deg \fG_{s_iv}(\bx)$, as pipe dreams for inverse permutations are related by transposition. Hence, we also have $\deg \fG_{s_iv}(\bx) \leq \deg \fG_v(\bx)$.

The lemma follows by recursion down two-sided weak order.
\end{proof}

\begin{example}
Note that the analogous result does not hold for the strong order.  For instance, $1432\leq 3412$ but $\deg \fG_{1432}(\bx)=5$ and $\deg \fG_{3412}(\bx)=4$.
\end{example}

\begin{lemma}\label{lem:deg_dominant}
	If $w$ is a dominant permutation, then 
	\[
	\deg \fCM_w(\bx) = \deg \fS_w(\bx)=\inv(w) = \raj(w).
	\]
\end{lemma}
\begin{proof}
From Proposition~\ref{prop:dominant_background}, we have $\fG_w(\bx) = \fS_w(\bx)$ when $w$ is dominant, establishing the first equality. The second equality is a standard fact about Schubert polynomials.
	The final equality was established in Proposition~\ref{prop:raj=inv}.
\end{proof}

Now  we turn to determining $\deg \fCM_w(\bx)$ for layered permutations $w$.

 \begin{lemma}\label{lem:deg_ealpha}
 	If $w= e_\alpha$, then $\deg \fCM_{w}(\bx) = \raj(w)$.
 \end{lemma}

\begin{proof}
Let $\alpha = (\alpha_1, \dots, \alpha_k)$ and let $w = e_\alpha$. It is convenient to abbreviate $\beta_m = \alpha_1+\alpha_2+ \cdots + \alpha_m$.
We will compute $\fCM_w(\bx)$ using the pipe dream formula (Theorem~\ref{thm:pipedreams}); note that only pipe dreams with the maximal number of crossings contribute to the highest degree part of $\fG_w(\bx)$. 
We will show that, in fact, there is a unique pipe dream for $e_{\alpha}$ with a maximal number of crossings. 

Specifically, let $Q_{\alpha}$ be the pipe dream which has bumping tiles on the antidiagonals $i+j-1 = \beta_m$, for $1 \leq m \leq k$, and crossing tiles on all antidiagonals $i+j-1=\gamma$ for $\gamma \not\in \{ \beta_1, \beta_2, \ldots, \beta_k \}$.
See  Example~\ref{ex:e_alpha} for an example. It is easy to check that $Q_{\alpha}$ is a pipe dream for $e_{\alpha}$.

Now, let $P$ be any pipe dream for $e_{\alpha}$. We will show that the set of crossing tiles in $P$ is a subset of the crossing tiles of $Q_{\alpha}$. Indeed, suppose for the sake of contradiction that $P$ has a crossing tile on the antidiagonal $i+j-1 = \beta_m$ for some $\beta_m$. Then the pipe dream $P$ corresponds to a permutation which is greater than the simple transposition $s_{\beta_m}$ in Bruhat order. But, in fact, $e_{\alpha} \not\geq s_{\beta_m}$, a contradiction.

We have thus shown that every pipe dream for $w$ is a subset of $Q_{\alpha}$, so the only maximal degree term in $\fCM_{e_{\alpha}}(\bx)$ is the term corresponding to $Q_{\alpha}$.
The antidiagonal $i+j-1 = \gamma$ has $\gamma$ crossing tiles on it, for each $\gamma \not\in \{ \beta_1, \ldots, \beta_k \}$, and these $\gamma$'s are precisely the descents of $e_{\alpha}$. 
So $\deg Q_{\alpha} = \maj(e_{\alpha})$. Since $e_{\alpha}$ is fireworks, we have $\raj(e_{\alpha}) = \maj(e_{\alpha})$.
\end{proof} 
 
\begin{example}\label{ex:e_alpha}
If $\alpha = (2,2,3,2)$, then $e_\alpha$ has one-line notation $214376598$. Below, we draw the Rothe diagram (on the left) and the pipe dream $Q_{\alpha}$ (on the right).
 
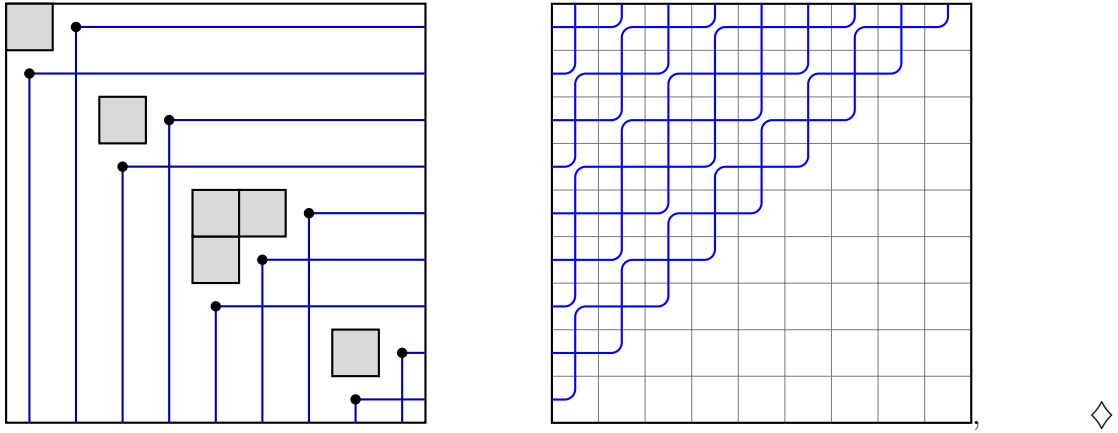
\[
  \begin{tikzpicture}[x=1.5em,y=1.5em]
      \draw[color=black, thick](0,1)rectangle(9,10);
     \filldraw[color=black, fill=gray!30, thick](0,9)rectangle(1,10);
     \filldraw[color=black, fill=gray!30, thick](2,7)rectangle(3,8);
     \filldraw[color=black, fill=gray!30, thick](4,5)rectangle(5,4);
     \filldraw[color=black, fill=gray!30, thick](4,6)rectangle(5,5);
     \filldraw[color=black, fill=gray!30, thick](5,6)rectangle(6,5);
     \filldraw[color=black, fill=gray!30, thick](7,3)rectangle(8,2);
     \draw[thick,darkblue] (9,9.5)--(1.5,9.5)--(1.5,1);
     \draw[thick,darkblue] (9,8.5)--(0.5,8.5)--(0.5,1);
     \draw[thick,darkblue] (9,7.5)--(3.5,7.5)--(3.5,1);
     \draw[thick,darkblue] (9,6.5)--(2.5,6.5)--(2.5,1);
     \draw[thick,darkblue] (9,5.5)--(6.5,5.5)--(6.5,1);
	 \draw[thick,darkblue] (9,4.5)--(5.5,4.5)--(5.5,1);
	 \draw[thick,darkblue] (9,3.5)--(4.5,3.5)--(4.5,1);
	 \draw[thick,darkblue] (9,2.5)--(8.5,2.5)--(8.5,1);
	 \draw[thick,darkblue] (9,1.5)--(7.5,1.5)--(7.5,1);
     \filldraw [black](1.5,9.5)circle(.1);
     \filldraw [black](0.5,8.5)circle(.1);
     \filldraw [black](3.5,7.5)circle(.1);
     \filldraw [black](2.5,6.5)circle(.1);
     \filldraw [black](6.5,5.5)circle(.1);
     \filldraw [black](5.5,4.5)circle(.1);
     \filldraw [black](4.5,3.5)circle(.1);
     \filldraw [black](8.5,2.5)circle(.1);
     \filldraw [black](7.5,1.5)circle(.1);
     \end{tikzpicture} \qquad \qquad
\begin{tikzpicture}[x=1.5em,y=1.5em]
\draw[step=1,gray, thin] (0,0) grid (9,9);
\draw[color=black, thick](0,0)rectangle(9,9);
\draw[thick,rounded corners,color=blue](0,8.5)--(1.5,8.5)--(1.5,9);
\draw[thick,rounded corners,color=blue] (0,7.5)--(0.5,7.5)--(0.5,9);
\draw[thick,rounded corners,color=blue] (0,6.5)--(1.5,6.5)--(1.5,8.5)--(3.5,8.5)--(3.5,9);
\draw[thick,rounded corners,color=blue] (0,5.5)--(0.5,5.5)--(0.5,7.5)--(2.5,7.5)--(2.5,9);
\draw[thick,rounded corners,color=blue] (0,4.5)--(2.5,4.5)--(2.5,7.5)--(5.5,7.5)--(5.5,9);
\draw[thick,rounded corners,color=blue] (0,3.5)--(1.5,3.5)--(1.5,6.5)--(4.5,6.5)--(4.5,9);
\draw[thick,rounded corners,color=blue] (0,2.5)--(0.5,2.5)--(0.5,5.5)--(3.5,5.5)--(3.5,8.5)--(6.5,8.5)--(6.5,9);
\draw[thick,rounded corners,color=blue] (0,1.5)--(1.5,1.5)--(1.5,3.5)--(3.5,3.5)--(3.5,5.5)--(5.5,5.5)--(5.5,7.5)--(7.5,7.5)--(7.5,9);
\draw[thick,rounded corners,color=blue] (0,0.5)--(0.5,0.5)--(0.5,2.5)--(2.5,2.5)--(2.5,4.5)--(4.5,4.5)--(4.5,6.5)--(6.5,6.5)--(6.5,8.5)--(8.5,8.5)--(8.5,9);
\end{tikzpicture}, \qedhere
\]
 \end{example}

 \begin{remark}
 	In fact, the proof of Lemma~\ref{lem:deg_ealpha} shows that there is a unique pipe dream for $e_{\alpha}$ of top degree, so $\fCM_{e_\alpha}(\bx)$ is a single monomial.
 \end{remark}

\begin{proposition}\label{prop:dominantMaxShape}
	Let $w \in S_n$ have shape $\alpha$. The permutation $w$ is $\leq_{LR}$-maximal among permutations of shape $\alpha$ if and only if $w$ is dominant.
\end{proposition}
\begin{proof}
($\Rightarrow$) Suppose $w \in S_n$ has shape $\alpha$ and is not dominant. We show it is not $\leq_{LR}$-maximal among permutations of that shape.

	Recall that a permutation $u$ is dominant if and only if there do \textbf{not} exist $i_1 < i_2 < i_3$ and $j_1 < j_2 < j_3$ with $u(i_1) = j_1$, $u(i_2) = j_3$, $u(i_3) = j_2$. We will call such $((i_1, i_2, i_3), (j_1, j_2, j_3))$ a \newword{$132$-pattern}.  We define the \newword{magnitude} of a $132$-pattern $((i_1, i_2, i_3), (j_1, j_2, j_3))$ to be the positive integer $(i_2-i_1) + (j_2-j_1)$. 

Consider a $132$-pattern $((i_1, i_2, i_3), (j_1, j_2, j_3))$ of minimal magnitude among all $132$-patterns of $w$.
Let $j_1 \in \pi_{d_1}(w)$, $j_2 \in \pi_{d_2}(w)$ and $j_3 \in \pi_{d_3}(w)$. Since $(i_1, j_1)$ is strictly northwest of both $(i_2, j_3)$ and $(i_3, j_2)$ in the diagram, we have $d_1 < d_2$ and $d_1 < d_3$. 

Without loss of generality, we will assume that $d_2 \leq d_3$ and show that $w$ is not $\leq_R$-maximal. (If we had made the opposite assumption that $d_2 \geq d_3$, the same argument {\it mutatis mutandis} would show that $w$ is not $\leq_L$-maximal.)

Set $i = i_2 - 1$ and $j = w(i)$. If $j>j_2$, then $((i_1, i, i_3), (j_1, j_2, j))$ is a $132$-pattern in $w$ with magnitude strictly less than $((i_1, i_2, i_3), (j_1, j_2, j_3))$, a contradiction. Also, clearly $i \neq i_3$, so $j \neq j_2$. Hence $j<j_2$. Thus, $w s_i >_R w$. We will therefore be done if we show that $w s_i$ has the same shape as $w$. 

Let $j \in \pi_d(w)$. Since $(i,j)$ is strictly northwest of $(i_3, j_2)$ in the diagram for $w$, we have $d < d_2 \leq d_3$.
If $d \leq d_3 -2$, then $w$ and $w s_i$ automatically have the same shape, since in this case the longest increasing run of $w$ starting at $j$ does not involve $j_3$. So it remains to analyze the case where $d = d_3-1$.

Assume therefore that $d=d_3-1$. Then we must have $d_2 = d_3$. Then observe that $(i,j)$ is strictly northwest of both $(i_2, j_3)$ and $(i_3, j_2)$ in the diagram for $w$, and that $j_2$ and $j_3$ are both in $\pi_{d_3}(w)$. Therefore, $ws_i$ has an increasing substring starting with $j$ of the same length as the maximal such run of $w$, since we may merely prepend $j$ onto a maximal increasing substring of either permutation starting with $j_2$.

($\Leftarrow$) Suppose $w$ is dominant and that $ws_i >_R w$. We will show that $\alpha(w) \neq \alpha(ws_i)$. Since $\alpha(w) = \alpha(w^{-1})$ and inverses of dominant permutations are also dominant, an equivalent argument shows that left covers of $w$ also have different shapes from $\alpha(w)$.  

Let $j = w(i)$ and $j' = w(i+1)$. By assumption $j < j'$. Since $w$ is dominant, $j'$ is the least integer greater than $j$ that is not $w(k)$ for some $k < i$. Therefore, any maximal increasing subsequence of $w$ starting with $j$ must include the letter $j'$, and indeed must start $jj'$. Let $j \in \pi_d(w)$. Then $j' \in \pi_{d+1}(w)$. 

If we have $k \in \pi_{d+1}(w)$, then $k \in \pi_{d+1}(ws_i)$. This is because the maximal increasing subsequences of $w$ starting with $k$ are the same as the maximal increasing subsequences of $ws_i$ starting with $k$.
On the other hand, $j \notin \pi_d(ws_i)$ since every maximal increasing subsequence of $w$ starting with $j$ must start $jj'$, but $j'$ does not follow $j$ in $ws_i$. We can however find increasing runs of $ws_i$ starting with $j$ that are of length one smaller, so $j \in \pi_{d+1}(ws_i)$. Therefore, $\alpha_{d+1}(w) < \alpha_{d+1}(ws_i)$ and so $\alpha(w) \neq \alpha(ws_i)$.
\end{proof}

\begin{theorem} \label{thm:raj=deg}
	Let $w \in S_n$. Then $\deg \fCM_w(\bx) = \raj(w)$.
\end{theorem}
\begin{proof}
	Suppose the shape of $w$ is $\alpha$. By Proposition~\ref{prop:dominantMaxShape}, there is a dominant permutation $u$ of shape $\alpha$ with $w \leq_{LR} u$. 
By Lemma~\ref{lem:fireworks_factorizations}, the layered permutation $e_\alpha$ satisfies $e_{\alpha} \leq_{LR} w$.

	Let $d = \raj(e_\alpha)$. By Lemma~\ref{lem:deg_ealpha}, we have $d = \deg \fCM_{e_\alpha}(\bx)$. Since $\alpha(u) = \alpha$, we have $\raj(u) = d$. 
	By Lemma~\ref{lem:deg_dominant}, we have $\fCM_u(\bx) = \raj(u) = d$. 
		
	Since $\raj(\bullet)$ and $\deg \fCM_{\bullet}(\bx)$ are weakly increasing with respect by $\leq_{LR}$ (by Corollary~\ref{cor:LR_raj_ineq} and Lemma~\ref{lem:divided_diffs}), we must also have $\raj(w) = d$ and $\deg \fCM_w(\bx) = d$.
	\end{proof}

\subsection{Matrix Schubert varieties of maximal Castelnuovo--Mumford regularity}

As an application of the preceding ideas, we will determine the permutations $w \in S_n$ which maximize $\raj(w) - \inv(w)$. By Theorem~\ref{thm:raj=deg}, these are the permutations whose matrix Schubert varieties are of maximal Castelnuovo--Mumford regularity. 

\begin{theorem} \label{thm:MaxReg}
Let $n$ be a positive integer and define $k$ by $\binom{k}{2} \leq n \leq \binom{k+1}{2}$; if $n$ is a triangular number, then we may choose either value for $k$. Then the permutations in $S_n$ which achieve the highest value of $\raj(w) - \inv(w)$ are the layered permutations $e_{\alpha_1 \alpha_2 \cdots \alpha_k}$ where $j-1 \leq \alpha_j \leq j$ and $\sum \alpha_j = n$; subject to the convention that we permit $\alpha_1 = 0$ and treat $e_{0 \alpha_2 \alpha_3 \cdots \alpha_k}$ as meaning $e_{\alpha_2 \alpha_3 \cdots \alpha_k}$. All of these permutations achieve $\raj(w) - \inv(w) = \binom{n+1}{2} - kn + \binom{k+1}{3}$. 
\end{theorem}

\begin{remark}
If $n$ is a triangular number, $n = \binom{m}{2}$, then there is a unique permutation which maximizes $\raj(w) - \inv(w)$, namely $e_{12 \cdots (m-1)} = e_{012 \cdots (m-1)}$. In general, if $n = \binom{k}{2} + j$ for $0\leq  j \leq k$, there are $\binom{k}{j}$ permutations which achieve the bound.
\end{remark}

\begin{remark}
The sequence of values $\binom{n+1}{2} - kn + \binom{k+1}{3}$ starts 0, 0, 1, 2, 4, 7, 10, 14, 19, 25, 31, 38, 46, 55, 65, 75 \dots. 
The differences between consecutive terms of this sequence are 0, 1, 1, 2, 3, 3, 4, 5, 6, 6, 7, 8, 9, 10, 10 \dots; in other words, the increasing sequence where each triangular number occurs twice and all other positive integers occur once.
The sequence $\binom{n+1}{2} - kn + \binom{k+1}{3}$ (with the initial zeroes deleted) is \cite[A023536]{OEIS}.
\end{remark}

\begin{proof}[Proof of Theorem~\ref{thm:MaxReg}] 
Let $w \in S_n$ and say it has shape $\alpha$. By Lemma~\ref{lem:fireworks_factorizations}, take a length-additive factorization $w = u e_{\alpha} v$ with $\raj(w) = \raj(e_{\alpha})$. Then $\raj(w) - \inv(w) = \raj(e_{\alpha}) - \inv(u) - \inv(e_{\alpha}) - \inv(v) \geq \raj(e_{\alpha}) - \inv(e_{\alpha})$, with equality if and only if $u=v=1$. So the maximum regularity will occur only for layered permutations.

Now, let $\alpha = (\alpha_1, \alpha_2, \ldots, \alpha_{\ell})$ be a composition of $n$.  We compute $\raj(e_{\alpha}) - \inv(e_{\alpha})$. The unique maximal pipe dream for $e_{\alpha}$ is described in the proof of Lemma~\ref{lem:deg_ealpha} and depicted in Example~\ref{ex:e_alpha}. The number of crosses in it is $\binom{n+1}{2} - \sum_{i=1}^{\ell} (\alpha_1 + \alpha_2 + \cdots + \alpha_i)$; the term $(\alpha_1 + \alpha_2 + \cdots + \alpha_i)$ being the number of elbows on the antidiagonal between the $i$th and the $(i+1)$st block.  Meanwhile, the length of $e_{\alpha}$ is $\inv(e_\alpha)=\sum_{j=1}^{\ell} \binom{\alpha_j}{2}$. So
\[ \raj(e_{\alpha}) - \inv(e_{\alpha}) = \binom{n+1}{2} - \sum_{i=1}^{\ell}  (\alpha_1 + \alpha_2 + \cdots + \alpha_i) - \sum_{j=1}^\ell \binom{\alpha_j}{2} . \]
We note for future reference that this formula would still be correct if we prefixed zeroes to our composition vector $\alpha$. 

We now rearrange the formula. We have $ \sum_{i=1}^{\ell}  (\alpha_1 + \alpha_2 + \cdots + \alpha_i) = \sum_{j=1}^{\ell+1} (\ell-j) \alpha_j$. 
So
\begin{align*}
	\raj(e_{\alpha}) - \inv(e_{\alpha}) &= \binom{n+1}{2} - \sum_{j=1}^{\ell} \left( (\ell+1-j) \alpha_j + \binom{\alpha_j}{2} \right) \\
&= \binom{n+1}{2} - \sum_{j=1}^{\ell} \left( \ell \alpha_j + \binom{\alpha_j-j+1}{2} - \binom{j}{2} \right) \\
&=  \binom{n+1}{2} - \ell \sum_{j=1}^{\ell} \alpha_j  + \sum_{j=1}^{\ell} \binom{j}{2} - \sum_{j=1}^{\ell} \binom{\alpha_j-j+1}{2}  \\
&= \binom{n}{2} - \ell n + \binom{\ell+1}{3} - \sum_{j=1}^{\ell} \binom{\alpha_j-j+1}{2}  \qquad (\ast)
\end{align*}

It is now easy to see that the claimed compositions achieve $\raj(e_{\alpha}) - \inv(e_{\alpha}) = \binom{n+1}{2} - kn + \binom{k+1}{3}$: For these compositions, we have $\alpha_j-j+1 \in \{ 0,1 \}$ for all $j$, so the final sum is $0$. (In the cases where $\alpha_1=0$, we have used that our formula is valid with an initial $0$ prepended to $\alpha$.)

We now need to show that no other composition can achieve as high a value.
It is tempting to think that we can simply say that we are done because $\binom{\alpha_j-j+1}{2} \geq 0$. However, we do not know that $\ell$ (the number of parts of our composition) is equal to $k$ (the index of the largest triangular number below $n$). We take a different route. 

Let $(\alpha_1, \alpha_2, \ldots, \alpha_{\ell})$ be a composition which maximizes $(\ast)$, and we emphasize that we have not prefixed any $0$'s, so all the $\alpha_i$ are strictly positive. 
Replacing $\alpha$ by $(\alpha_1-1, \alpha_2, \ldots, \alpha_i+1, \ldots, \alpha_{\ell})$ and subtracting the resulting values of $(\ast)$, we deduce that $\alpha_i - i \geq \alpha_1-2$ so $\alpha_i \geq i+\alpha_1-2 \geq i-1$. On the other hand, if $\alpha_i>1$ then we may replace $\alpha$ by $(1, \alpha_1, \alpha_2, \ldots, \alpha_i-1, \ldots, \alpha_n)$; subtracting the resulting values for $(\ast)$ gives $\alpha_i-(i+1) \leq 0$ so $\alpha_i \leq i+1$ and, of course, if $\alpha_i=1$ then we also have $\alpha_i \leq i+1$. In short, we see that $i-1 \leq \alpha_i \leq i+1$. 

We now claim that it is impossible that we both have $\alpha_i = i-1$ and $\alpha_j = j+1$ for some indices $i$ and $j$. Indeed, if this occurred, then replacing $\alpha$ with $(\alpha_1, \alpha_2, \ldots, \alpha_i+1, \ldots, \alpha_j-1, \ldots, \alpha_{\ell})$ would keep the $i$th summand in $(\ast)$ the same and decrease the $j$th summand from $1$ to $0$, making $(\ast)$ larger. So it either holds that $i-1 \leq \alpha_i \leq i$ or that $i \leq \alpha_i \leq i+1$ for all $i$. If we are in the latter case, prepend a $0$ to $\alpha$. The modified $\alpha$ will then have $i-1 \leq \alpha_i \leq i$ for all $i$. We will now use $\ell$ to refer to the number of parts in this modified $\alpha$.

We have shown that the optimal $\alpha$ has $i-1 \leq \alpha_i \leq i$ for all $i$, and possibly an initial $0$, and we have excluded the case $(0,1,2,\ldots,\ell-1)$ by fiat. Then \[
n = \sum_{i=1}^{\ell} \alpha_i \leq \sum_{i=1}^{\ell} i = \binom{\ell+1}{2}
\]  and 
\[
n = \sum_{i=1}^{\ell} \alpha_i \geq \sum_{i=1}^{\ell} (i-1) = \binom{\ell}{2},
\] so $\binom{\ell}{2} \leq n \leq \binom{\ell+1}{2}$, and $\ell$ is the $k$ in the statement of the theorem (or one of the two values for $k$, if $n$ is a triangular number).
\end{proof}

\section{Rajchgot polynomials}
\label{sec:rajchgotpolynomials}

We identify an important family of polynomials, indexed by set partitions.

\begin{definition}
Let $\pi$ be a set partition of $[n]$ and let $w$ be the unique fireworks permutation with $\pi(w) = \pi$. 
We define the  \newword{Rajchgot polynomial} $\fR_{\pi}(\bx)$ to be $\fCM_w(\bx)$.
\end{definition}

The goal of this section is to prove the following factorization theorem for double Castelnuovo--Mumford polynomials, establishing the first part of Theorem~\ref{thm:main}.

\begin{theorem}\label{thm:factorization}
For any $w \in S_n$,
\[ \fCM_w(\bx;\by) = \fR_{\pi(w)}(\bx) \fR_{\pi(w^{-1})}(\by). \]
\end{theorem}

See Example~\ref{eg:posetExample} for examples of this Theorem.

\begin{lemma}\label{lem:demazure_fireworks}
Let $p \in [\Phi_{\rm inv}(w), w]_L$ and $q \in [\Phi(w), w]_R$. Then $q \ast p = w$ if and only if $p = \Phi_{\rm inv}(w)$ and $q = \Phi(w)$.
\end{lemma}
By Lemma~\ref{lem:fireworks_factorizations}, every permutation $w \in S_n$ with $\alpha(w) = \alpha$ has a unique length-additive factorization $w = ue_\alpha v$, where $\Phi(w) = ue_\alpha$ and $\Phi_{\rm inv}(w) = e_\alpha v$. Therefore, the interval $[\Phi_{\rm inv}(w), w]_L$ is $\{ u' e_{\alpha} v : \id \leq_L u' \leq_L u \}$ and $[\Phi(w), w]_R$ is $\{ u e_{\alpha} v' : \id \leq_R v' \leq_R v \}$.

\begin{proof}
Let $w$ have shape $\alpha$ and
uniquely write $w = u e_{\alpha} v$ as above, so $p = u' e_{\alpha} v$ and $q = u e_{\alpha} v'$ for some $u' \in [\id,u]_L$ and some $v' \in [\id,v]_R$. 

If $p = \Phi_{\rm inv}(w)$ and $q = \Phi(w)$, then $q \ast p = u e_{\alpha} \ast e_{\alpha} v = u e_{\alpha} v = w$ as desired, since the layered permutation $e_\alpha$ is idempotent for the Demazure product $\ast$.

Now, suppose that at least one of $u'$ and $v'$ is not the identity; we must show that $(u e_{\alpha} v') \ast (u' e_{\alpha} v)$ is not equal to $u e_{\alpha} v$. 
Since the products in parentheses are length additive, we have $(u e_{\alpha} v') \ast (u' e_{\alpha} v) = u \ast e_{\alpha} \ast v' \ast u' \ast e_{\alpha} \ast v$.

Since at least one of $u'$ and $v'$ is not the identity, at least one of $e_{\alpha} \ast v'$ and $u' \ast e_{\alpha}$ is longer than $e_{\alpha}$ and thus $(e_{\alpha} \ast v') \ast (u' \ast e_{\alpha})$ is not $e_{\alpha}$. 
However, $e_{\alpha} \ast v' \ast u' \ast e_{\alpha}$ is of the form $e_{\alpha} \ast x \ast e_{\alpha}$, so Lemma~\ref{lem:ealpha_lemma} applies and we see that $\raj(e_{\alpha} \ast v' \ast u' \ast e_{\alpha}) > \raj(e_{\alpha})$.
Using Corollary~\ref{cor:LR_raj_ineq}, we then have $\raj(u \ast e_{\alpha} \ast v' \ast u' \ast e_{\alpha} \ast v) > \raj(e_{\alpha})$. 
However, we also have $\raj(e_{\alpha}) = \raj(u e_{\alpha} v)$, so we deduce that $u \ast e_{\alpha} \ast v' \ast u' \ast e_{\alpha} \ast v \neq u \ast e_{\alpha} \ast v$, as required.
\end{proof}

\begin{proof}[Proof of Theorem~\ref{thm:factorization}]
Recall the Cauchy identity for double Grothendieck polynomials (see \cite[Proof of Theorem~6.7]{Lenart.Robinson.Sottile}) is
\begin{equation}\label{eq:Cauchy}
\fG_w(\bx;\by) = \sum_{q \ast p = w} (-1)^{\inv(w) - \inv(p) - \inv(q)} \fG_p(\bx) \fG_{q^{-1}}(\by).
\end{equation}
We note that the condition $q \ast p = w$ implies that $p \leq_L w$ and $q \leq_R w$. 

Now, let's strip off the terms of degree $(\raj(w), \raj(w))$. These can only come from $(p,q)$ with $\raj(p) = \raj(q) = \raj(w)$. By combining Corollary~\ref{cor:dominates_raj} with Lemmas~\ref{lem:cover_blob} and~\ref{lem:cover_blob_left}, the conditions $p \leq_L w$ and $\raj(p) = \raj(w)$ are collectively equivalent to $p \in [\Phi_{\rm inv}(w), w]_L$. Similarly, we need $q \in [\Phi(w), w]_R$. So the highest degree parts of Equation~\eqref{eq:Cauchy} are also the highest degree parts of
\[ \sum_{\substack{q \ast p = w\\ p \in [\Phi_{\rm inv}(w), w]_L\\ q \in [\Phi(w), w]}} \fG_{p}(\bx) \fG_{q^{-1}}(\by) . \]
But now Lemma~\ref{lem:demazure_fireworks} says that the single term satisfying the summation conditions is $(p,q) = (\Phi_{\rm inv}(w), \Phi(w))$. 
So the highest degree part of $\fG_w(\bx;\by)$ is also the highest degree part of $\fG_{\Phi_{\rm inv}(w)}(\bx) \fG_{\Phi(w)^{-1}}(\by) = \fR_{\pi(w)}(\bx) \fR_{\pi(w^{-1})}(\by)$, as desired.
\end{proof}

We are now ready to prove another portion of Theorems~\ref{thm:degree} and~\ref{thm:main}, namely that the leading monomial of the double Grothendieck polynomial is at most $\bx^{\rajc(w)} \by^{\rajc(w^{-1})}$. The verification that there is in fact a monomial with this degree will occur in Section~\ref{sec:pushingPluses}.

For the following argument, it will be useful to consider pipe dreams as a special case of more general planar histories, confined to a rectangular region (cf.\ \cite{Fomin.Kirillov,Weigandt}).  A \newword{southwest planar history of size $n$} is a configuration of $n$ paths, each of which starts at the top edge of a rectangle and ends at the left edge.  Within the rectangle, paths move weakly southwest at all times.  Paths may cross, but they do not travel concurrently at any point, and there are no triple crossings.

To each crossing, we associate a simple reflection as follows.  Let $k$ be the number of paths which pass strictly southeast of the crossing.  Then we label the crossing with $s_{n-k-1}$.  To produce a word, order crossings from top to bottom, breaking ties within rows by reading from right to left.  The reader may verify that in the case of pipe dreams, this procedure produces the same reading word as the one we described in Section~\ref{GrothBackground}.
We associate a permutation to a planar history by taking the Demazure product of its reading word.

\begin{Theorem} \label{thm:exponentbound}
Let $w$ be a permutation, let $\rajc(w) = (r_1, \ldots, r_n)$ and let $\rajc(w^{-1}) = (s_1, \ldots, s_n)$. For any term order satisfying $x_1 < x_2 < \cdots < x_n$ and $y_1 < y_2 < \cdots < y_n$, every monomial of $\fCM_w(\bx; \by)$ is at most $x_1^{r_1} \cdots x_n^{r_n} y_1^{s_1} \cdots y_n^{s_n}$. 
\end{Theorem}

\begin{proof}
Let $x_1^{a_1} \cdots x_n^{a_n} y_1^{b_1} \cdots y_n^{b_n}$ be a monomial occurring in $\fCM_w(\bx;\by)$, and let $P$ be a corresponding pipe dream, so that $P$ has $a_i$ crosses in row $i$ and $b_j$ crosses in column $j$. We must show, for all $k$, that $a_{k+1} + \cdots + a_n \leq r_{k+1} + \cdots + r_n$ and $b_{k+1} + \cdots + b_n \leq s_{k+1} + \cdots + s_n$. We will prove that $a_{k+1} + \cdots + a_n \leq r_{k+1} + \cdots + r_n$; the other inequality is analogous.
When $k=0$, this says that the number of crosses in any pipe dream for $w$ is bounded by $\raj(w)$, which we showed in Theorem~\ref{thm:raj=deg}. We will now provide the proof for general $k$.

Take the pipe that exits in the first row of $P$ and delete it to create a new diagram $Q^{1}$.  This is now a southwest planar history of size $n-1$.
Now, take the pipe that exits in the second row of $Q^1$ and delete it to create $Q^{2}$, a southwest planar history of size $n-2$.  Continue in this way to define planar histories $Q^{k}$ for each $1 \leq k \leq n-1$.

First, observe that the permutation corresponding to $Q^{k}$ is the unique $w^k\in S_{n-k}$ such that $w^k_1\cdots w^k_{n-k}$ have the same relative order as $w_{k+1}\cdots w_n$.  In particular, this implies $\raj(w^k)=r_{k+1}+\dots +r_n$.   
We want to show that $a_{k+1} + \cdots + a_n \leq \raj(w^k)$. 

Let $P^{k}$ be the result of restricting $P$ to rows $k+1$ through $n$ (numbering from the top) and let $u^{k}$ be the permutation for which $P^{k}$ is a pipe dream.  By Theorem~\ref{thm:raj=deg}, $a_{k+1}+\dots +a_{n}\leq \raj(u^k)$. 
Observe that $P^{k}$ is contained within $Q^{k}$, each has $n-k$ paths, and the crossings have the same associated simple reflections in $P^k$ as they do in $Q^k$.  In particular, it follows that $u^k \leq_L w^k$. Then, by Corollary~\ref{cor:LR_raj_ineq}, we have 
\[a_{k+1}+\dots +a_{n}\leq \raj(u^k) \leq \raj(w^k),\]
as desired.
\end{proof}

\subsection{A divided difference recurrence for Rajchgot polynomials}

The material in this section is not used in the rest of the paper, but gives an efficient way to compute Rajchgot polynomials.

For each set partition $\pi$, there is a unique inverse fireworks permutation $w$ with $\pi(w)$ equal to $\pi$, and we have $\fR_{\pi}(\bx) = \fCM_w(\bx)$. 
Thus, if we want to compute all the Rajchgot polynomials without redundancy, we should compute $\fCM_w(\bx)$ for $w$ ranging over inverse fireworks permutations.
By Lemma~\ref{lem:maximal_is_valley}, the set of inverse fireworks permutations of shape $\alpha$ is the right interval $[e_{\alpha}, f_{\alpha}]_R$, where $e_\alpha$ is the layered permutation of shape $\alpha$ and $f_\alpha$ is the valley permutation of shape $\alpha$. 
We describe a method for computing the Rajchgot polynomials of all permutations in this interval by starting at the valley permutation $f_{\alpha}$ and walking down in right weak order.
Our base case, the Rajchgot polynomial for $f_{\alpha}$ is simple.
\begin{prop} \label{prop:raj_base_case}
Let $\alpha = (\alpha_t, \alpha_{t+1}, \ldots, \alpha_n)$ be a composition and $f_{\alpha}$ the corresponding valley permutation. Let $\sigma_k = \alpha_t+\alpha_{t+1} + \cdots + \alpha_k$ (so $\sigma_n=n$) and let \[
[n] \setminus \{ \sigma_t, \sigma_{t+1},\ldots, \sigma_n \} = \{ \rho_1 < \rho_2 < \cdots < \rho_t \}.
\]
 Then
\[ \fR_{\pi(f_{\alpha})} = \prod x_j^{\rho_j} . \]
\end{prop}

\begin{proof}[Proof sketch]
The permutation $f_{\alpha}$ is dominant, so $\fR_{\pi(f_{\alpha})} = \bx^{\rajc(f_{\alpha})}$. We leave computing the Rajchgot code of $f_{\alpha}$ as an exercise for the reader.
\end{proof}

We now discuss how Rajchgot polynomials transform when we go down by a cover in right weak order. 
Recall from Section~\ref{GrothBackground} the divided difference operator
\[  \overline{\partial}_i(f) = \partial_i((1-x_{i+1})f) = \frac{(1-x_{i+1}) f - (1-x_i) s_i\cdot f}{x_i - x_{i+1}} . \]
for $f \in \ZZ[x_1, x_2, \ldots, x_n]$.
Similarly define
 \[ \rN_i(f) = \frac{x_{i+1} f - x_i s_i\cdot f}{x_i - x_{i+1}} . \]
 Then by the recursive definition of Grothendieck polynomials, we have the following.
 \begin{prop} \label{prop:Raj_dd}
Let $w \in S_n$ and let $1 \leq i \leq n-1$. Then
\[ \rN_i(\fCM_w(\bx)) = 
\begin{cases}
\fCM_{w s_i}(\bx), & w s_i <_R w,\ \raj(w s_i) = \raj(w); \\
0, & w s_i <_R w,\ \raj(w s_i) < \raj(w); \\
- \fCM_w(\bx), & w s_i >_R w. \\
\end{cases}\]
 \end{prop}
 The reader who is confused about signs should recall that $\fCM_w(\bx)$ is the top degree part of $(-1)^{\deg \fG_w(\bx) - \inv(w)} \fG_w(\bx)$, and $\inv(w) = \inv(w s_i) \pm 1$.

Propositions~\ref{prop:raj_base_case} and~\ref{prop:Raj_dd} give a recursion for computing $\fCM_w(\bx)$ for every $w \in [e_{\alpha}, f_{\alpha}]$, and therefore for computing all Rajchgot polynomials of shape $\alpha$.

When carrying out this computation practically, it is most useful to encode set partitions by words as follows.
Given a set partition $(\pi_t, \pi_{t+1}, \ldots, \pi_n )$ of $[n]$ with $|\pi_k| = \alpha_k$, we define the corresponding word $p=p_1 p_2 \cdots p_n$ in $[n]^n$ by $p_j = k$ if and only if $j \in \pi_k$, and we write $\fR_p$ in place of $\fR_{\pi}$.
Not every word corresponds to a set partition, because we always order set partitions so that $\max(\pi_t) < \max(\pi_{t+1}) < \cdots < \max(\pi_{n-1}) < \max(\pi_n)=n$.

We note that $\pi(e_{\alpha})$ corresponds to the word $1^{\alpha_1} \cdots n^{\alpha_n}$, and that $\pi(f_{\alpha})$ corresponds to $n^{\alpha_n-1} (n-1)^{\alpha_{n-1}-1} \cdots (t+1)^{\alpha_{t+1}-1} t^{\alpha_t} (t+1) \cdots (n-1) n$. 
So our recursive procedure starts with $n^{\alpha_n-1} (n-1)^{\alpha_{n-1}-1} \cdots (t+1)^{\alpha_{t+1}-1} t^{\alpha_t} (t+1) \cdots (n-1) n$ and walks towards $1^{\alpha_1} \cdots n^{\alpha_n}$.
We let $S_n$ act on words of length $n$ by reordering the letters. 
In this notation, Proposition~\ref{prop:Raj_dd} translates as follows.
\begin{prop}
Let $p_1 p_2 \cdots p_n$ be a word corresponding to a set partition and let $1 \leq i \leq n-1$. Then 
\[ \rN(\fR_p(\bx)) = \begin{cases}
\fR_{p s_i}(\bx), & p_i > p_{i+1}; \\
0, & p_i < p_{i+1}; \\
- \fR_{p}(\bx), & p_i = p_{i+1}.
\end{cases} \]
\end{prop}

\begin{rem}
The condition that $\rN(f)=0$ is equivalent to ``$x_i$ divides $f$ and $f/x_i$ is symmetric in $x_i$ and $x_{i+1}$''; the condition $\rN(f)= -f$ is equivalent to ``$f$ is symmetric in $x_i$ and $x_{i+1}$." 
\end{rem}

\begin{example}
Let $\alpha$ be $1+1+2$, so we are studying anagrams of $2344$. There are $12$ anagrams of this string, but only three of these correspond to set partitions: $2344$, $2434$, and $4234$. These correspond to the permutations $1243$, $1423$, and $4123$. The first of these is the layered permutation $e_{\alpha}$ and the last is the valley permutation $f_{\alpha}$. 

In the diagram below, we have drawn the Hasse diagram of $[e_{\alpha}, f_{\alpha}]_R$, labeling each edge $(u,v)$ with the index $i$ such that $u = v s_i$, and labeling each node with the inverse fireworks permutation $w$, with the word $p$, and with its Rajchgot polynomial $\fR_w(\bx) = \fR_p(\bx)$. We have underlined the leading term, whose exponent is $\rajc(w) = p-(1,2,3,\ldots,n)$.
\[ 
\xymatrix{
\framebox[0.5 \textwidth]{$w = 4123$\quad $p = 4234$\quad $\fR_w(\bx) = \underline{x_1}^3$ } \ar@{-}[d]^1 \\
\framebox[0.5 \textwidth]{$w = 1423$\quad $p = 2434$\quad $\fR_w(\bx) = x_1^2 x_2+\underline{x_1 x_2^2}$ } \ar@{-}[d]^2 \\
\framebox[0.5 \textwidth]{$w = 1243$\quad $p = 2344$\quad $\fR_w(\bx) = \underline{x_1 x_2 x_3}$ } \\
}
\]
Compare this example to the bottom row in Example~\ref{eg:posetExample}.
\end{example}

\section{Maximal pipe dreams} \label{sec:pushingPluses}

The primary statement now outstanding from Theorems~\ref{thm:degree} and~\ref{thm:main} is that there is a pipe dream for $w$ with
$\rajc_i(w)$ crossing tiles in row $i$ and $\rajc_j(w^{-1})$ crossing tiles in column $j$, thereby contributing the 
 monomial $\bx^{\rajc(w)} \by^{\rajc(w^{-1})}$ to $\fCM_w(\bx;\by)$. We call such a pipe dream a \newword{maximal pipe dream} for $w$.
We will verify the existence of this pipe dream in this section. We also need to check that this monomial has coefficient $1$ in $\fCM_w(\bx;\by)$, which we will also verify in this section.

In other words, our goal is to prove:
\begin{Theorem} \label{thm:exponentachieved}
Let $w$ be a permutation, let $\rajc(w) = (r_1, \ldots, r_n)$ and let $\rajc(w^{-1}) = (s_1, \ldots, s_n)$.  
There is a pipe dream for $w$ with $r_i$ crosses in row $i$ and $s_j$ crosses in column $j$. 
\end{Theorem}

\begin{remark}
We find it frustrating that we do not have a direct recipe for the maximal pipe dream in terms of $w$. 
The reader who believes they know one should test it on the example $w=14523$. The unique maximal pipe dream for this permutation is shown below, but straightforward greedy procedures get stuck at local maxima with only five crosses:
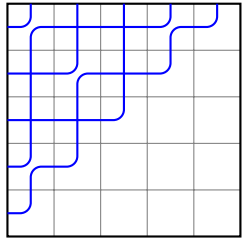
\[ \begin{tikzpicture}[x=1.5em,y=1.5em]
\draw[step=1,gray, thin] (0,0) grid (5,5);
\draw[color=black, thick](0,0)rectangle(5,5);
\draw[thick,rounded corners,color=blue] (0,4.5)--(0.5,4.5)--(0.5,5);
\draw[thick,rounded corners,color=blue] (0,3.5)--(1.5,3.5)--(1.5,5);
\draw[thick,rounded corners,color=blue] (0,2.5)--(2.5,2.5)--(2.5,5);
\draw[thick,rounded corners,color=blue] (0,1.5)--(0.5,1.5)--(0.5,4.5)--(3.5,4.5)--(3.5,5);
\draw[thick,rounded corners,color=blue] (0,0.5)--(0.5,0.5)--(0.5,1.5)--(1.5,1.5)--(1.5,3.5)--(3.5,3.5)--(3.5,4.5)--(4.5,4.5)--(4.5,5);
\end{tikzpicture} \]
\end{remark}

\subsection{Proof of Theorem~\ref{thm:exponentachieved}}

Let $\alpha$ be the shape of $w$, so that we have a length-additive factorization $w = u e_{\alpha} v$, as in Lemma~\ref{lem:fireworks_factorizations}. 
We already constructed a maximal pipe dream for $e_{\alpha}$ in Lemma~\ref{lem:deg_ealpha}. 
We will now show how to modify this maximal pipe dream for $e_\alpha$ to obtain a maximal pipe dream for $w$.

If $w \neq e_{\alpha}$, then at least one of $u$ and $v$ is not the identity. Without loss of generality, suppose that $u$ is not the identity. Let $i$ be such that $s_iu <_L u$ (i.e., $i$ is a \emph{left descent} of $u$). 
Put $u' = s_iu$ and $w' = u' e_{\alpha} v$. Since $w' \geq_{LR} e_{\alpha}$, the permutations $w$ and $w'$ both have shape $\alpha$ by Lemma~\ref{lem:interval_structure}. 
Assume inductively that we have a maximal pipe dream $P'$ for $w'$; we will show that there is a maximal pipe dream $P$ for $w$. 
We first need to investigate the structure of $P'$ in rows $i$ and $i+1$. 

In any given column of $P'$, rows $i$ and $i+1$ look like one of $\domzz$, $\domzp$, $\dompz$, or $\dompp$.
We define the \newword{simplified array} to be the $2$-row pipe dream obtained by deleting all $\dompp$ columns from rows $i$ and $i+1$, and denote the simplified array of $P'$ by $Q'$.
We consider $Q'$ to extend infinitely to the right, with all columns sufficiently to the right equal to $\domzz$.

\begin{lem} \label{lem:leftcolumn}
The left column of $Q'$ is $\domzp$.
\end{lem}

\begin{proof}
The simplified array, by construction, has no $\dompp$ columns. 
If the left column were $\dompz$, then the pipes starting in rows $i$ and $i+1$ would cross, contradicting that $s_i w' >_L w'$.

If the left column were $\domzz$, then replacing it with $\dompz$ would give a pipe dream for $w$ with one more cross than in $P'$. 
But, by assumption, $P'$ has $\raj(w')$ crosses, and $\raj(w) = \raj(w')$ since $w$ and $w'$ have the same shape, so no pipe dream for $w$ can have more than $\raj(w')$ crosses, a contradiction.
\end{proof}

\begin{lem} \label{lem:forbiddenQpatterns}
The following patterns cannot occur in consecutive columns of $Q'$: $\domzz\!\dompz$, $\domzp\!\domzz$, $\dompz\!\dompz$.
\end{lem}

\begin{proof}
If we had $\domzz\!\dompz$ or  $\domzp\!\domzz$, then replacing these columns with $\domzp\!\dompz$ would give a pipe dream for $w'$ with more crosses than $P'$. 

Now, suppose for the sake of contradiction that   $\dompz\!\dompz$ occurred. By Lemma~\ref{lem:leftcolumn}, it cannot occur in the left two columns, so we can travel to the left from this pattern until we first see a column which is not $\dompz$. At this point, we will have either $\domzz\!\dompz\!\dompz$ or else $\domzp\!\dompz\!\dompz$. The first we have ruled out in the first paragraph of this proof. If  $\domzp\!\dompz\!\dompz$ occurs, let $P''$ be the pipe dream obtained by replacing it with  $\domzp\!\domzp\!\dompz$. Then $P''$ is also a pipe dream for $w'$, and its monomial differs from that of $P'$ by a factor of $x_i x_{i+1}^{-1}$. So the monomial of $P''$ dominates the monomial of $P'$, contradicting that $P'$ is maximal.
\end{proof}

Combining Lemmas~\ref{lem:leftcolumn} and~\ref{lem:forbiddenQpatterns}, we see that the simplified array $Q'$ is of the form
\[ \domzp^{a_1} \dompz \domzz^{b_1} \domzp^{a_2} \dompz \domzz^{b_2} \cdots \domzp^{a_{k-1}} \dompz \domzz^{b_{k-1}}  \domzp^{a_k} \dompz \domzz^{\infty} \]
for some constants $a_1$, $a_2$, \dots, $a_k$, $b_1$, $b_2$, \dots, $b_{k-1}$, where the $a_j$ are positive integers and the $b_j$ are nonnegative integers.

The following lemma concludes the proof of Theorem~\ref{thm:exponentachieved}.

\begin{lem} \label{lem:pushpluses}
Define a pipe dream $P$ by replacing the columns $\domzp^{a_1} \dompz$ of the simplified array by $\dompz^{a_1+1}$, while leaving all other columns (including the $\dompp$ columns) as in $P'$. Then $P$ is a maximal pipe dream for $w$. Moreover, in the above notation, we have $a_2=a_3 = \cdots=a_k=1$ in $Q'$.
\end{lem}

\begin{proof}
It is easy to check that $P$ is a pipe dream for $w$. 

Let $A = \sum_{j=1}^k a_j$. Let $\ell$ be the number of $\dompp$ columns in rows $i$ and $i+1$ of $P'$. Then $P'$ has $k+\ell$ crosses in row $i$ and $A+\ell$ crosses in row $i+1$. 
The pipe dream $P$ has $a_1+k+\ell$ crosses in row $i$ and $A+\ell-a_1$ crosses in row $i+1$.

By our assumption that $P'$ is maximal, we have $\rajc_i(w') = k+\ell$ and $\rajc_{i+1}(w') = A+\ell$. 
By Lemma~\ref{lem:cover_rajcode}, we have $\rajc_i(w) = A+\ell+1$ and $\rajc_{i+1}(w) = k+\ell-1$, while $\rajc_j(w) = \rajc_j(w')$ for $j \not\in \{ i,i+1 \}$. 

By Theorem~\ref{thm:exponentbound}, the pipe dream $P$ must have at most \[
\rajc_{i+1}(w) + \rajc_{i+2}(w) + \cdots +\rajc_n(w)
\]
 crosses in the bottom $n-i$ rows. Since  $\rajc_j(w) = \rajc_j(w')$ for $j>i+1$, and $P$ and $P'$ agree in rows below $i+1$, this tells us that  $P$ must have at most  $\rajc_{i+1}(w) = k+\ell-1$ crosses in row $i+1$. So $A+\ell-a_1 \leq k+\ell-1$ or, in other words, $\sum_{j=2}^k a_j \leq k-1$. 

But each $a_j$ is at least $1$. So we must have $a_2 = a_3 = \cdots = a_k= 1$, and $P$ has exactly $\rajc_{i+1}(w)$ crosses in row $i+1$. Since $P$ and $P'$ match outside rows $i$ and $i+1$, we see that $P$ has $\rajc_j(w) = \rajc_j(w')$ crosses in every such row outside of $i$ and $i+1$. Finally, since we only moved crosses vertically, $P$ and $P'$ the same number of crosses in each column, namely $\rajc(w^{-1}) = \rajc((w')^{-1})$. This completes the proof that $P$ is maximal for $w$ and, along the way, we noted that $a_2 = a_3 = \cdots = a_k=1$. 
\end{proof}

\subsection{The leading monomial has coefficient one}

Finally, we will show that maximal pipe dreams are unique. 

\begin{prop} \label{prop:coeff1Raj}
Let $\pi$ be a set partition. Then the leading monomial of the Rajchgot polynomial $\fR_{\pi}(\bx)$ has coefficient $1$.
\end{prop}
\begin{proof}
Let $w$ be the inverse fireworks permutation with $\pi(w) = w$. 
So $\fR_{\pi}(\bx) = \fCM_w(\bx)$. Let $\rajc(w)= (r_1, \ldots, r_n)$. 
In the proof of Lemma~\ref{lem:pushpluses}, we constructed a pipe dream for $w$ with monomial $x_1^{r_1} x_2^{r_2} \cdots x_n^{r_n}$; we now will show that pipe dream is unique.

Let $P$ be a maximal pipe dream for $w$.
Fix an index $k$ and define permutations $u^k$ and $w^k$ in $S_{n-k}$ as in the proof of Theorem~\ref{thm:exponentbound}.
By the proof of Theorem~\ref{thm:exponentbound}, we must have $\raj(u^k) = \raj(w^k)$. 

The permutation $w^k$ is obtained from $w$ by applying $k$-fold application of the operation in Lemma~\ref{lem:fireworks_delete_position}.
By Lemma~\ref{lem:fireworks_delete_position}, $w^k$ is also inverse fireworks. 
But we showed in the proof of Theorem~\ref{thm:exponentbound} that $u^k \leq_L w^k$. So the only $v \in S_{n-k}$ with  $v \leq_L w^k$ and $\raj(v) =\raj(w^k)$ is $v=w^k$ itself.
So we have proved, for every $k$, that $u^k = w^k$.

Recall that $w^k$ is the permutation for which $P^k$ is a pipe dream.
So, for each $k$, we know where the pipes end at the bottom of row $k$ and where they end at the top of row $k$.
There is a unique maximal way to place crosses in row $k$ to connect up those pipes. 
So the permutation $w$ determines the list of permutations $w^1$, $w^2$, \dots, $w^n$, and this determines the positions of the crosses in each row of $P$. Thus $P$ is unique.
\end{proof}

\begin{Theorem} \label{thm:coeff1}
Let $w$ be a permutation in $S_n$. The monomial $\bx^{\rajc(w)} \by^{\rajc(w^{-1})}$ appears with coefficient $1$ in $\fCM_w(\bx; \by)$.
\end{Theorem}
\begin{proof}
	Immediate by combining Theorem~\ref{thm:factorization} with Proposition~\ref{prop:coeff1Raj}.
\end{proof}

\section*{Acknowledgements}
We are grateful to Jenna Rajchgot for bringing this problem to our attention and explaining how to think about Castelnuovo--Mumford regularity. AW is grateful for many helpful conversations with Colleen Robichaux. OP was partially supported by a Mathematical Sciences Postdoctoral Research Fellowship (\#1703696) from the National Science Foundation, as well as by a Discovery Grant (RGPIN-2021-02391) and Launch Supplement (DGECR-2021-00010) from the Natural Sciences and Engineering Research Council of Canada. DES was supported in part by National Science Foundation grants DMS-1600223, DMS-1855135 and DMS-1854225. AW was partially supported by Bill Fulton's Oscar Zariski Distinguished Professor Chair funds.

\appendix

\section{Partial derivatives of Grothendieck polynomials}\label{sec:derivative} The authors originally discovered many of this results in this paper by a different route, using a differential recurrence relation for Grothendieck polynomials that seems both new and interesting to us. In the end, this technique proved less powerful than the methods that we have used in this paper, but we record it here, as we hope it will have other uses.

We will use the following two differential operators on $\ZZ[x_1, x_2, \ldots, x_n]$: 
\[ E = \sum_{i=1}^n x_i \frac{\partial}{\partial x_i} \ \mbox{and} \ \nabla = \sum_{i=1}^n \frac{\partial}{\partial x_i} . \]

\begin{theorem} \label{thm:DerivRecursion}
For any permutation $w \in S_n$, we have 
\begin{equation}\label{eq:david}
(\maj(w^{-1}) + \nabla-E) \fG_w(\bx) = \sum_{s_k w <_L w} k \fG_{s_k w}(\bx).
\end{equation}
\end{theorem}

\begin{remark}
We restate Equation~\eqref{eq:david} using the notion of \newword{$\beta$-Grothendieck polynomials}~\cite{Fomin.Kirillov}.
The $\beta$-Grothendieck polynomial is the polynomial
\[
\fG_w^{(\beta)}(x_1,\dots, x_n) \coloneqq \frac{\fG_w(-\beta x_1, \dots, -\beta x_n)}{(-\beta)^{\inv(w)}}.
\] 
in $\ZZ[x_1, x_2, \ldots, x_n, \beta]$.
Define the operator
\[
\nabla^\beta = \nabla + \beta^2 \frac{\partial}{\partial \beta}.
\]
Then Equation~\eqref{eq:david} is equivalent to
\begin{equation}\label{eq:anna}
\nabla^\beta \fG_w^{(\beta)}(\bx) = \beta(\maj(w^{-1}) - \inv(w)) \fG^{(\beta)}_w(\bx) + \sum_{s_k w <_L w} k \fG^{(\beta)}_{s_k w}(\bx).
\end{equation}
By setting $\beta=0$ in Equation~\eqref{eq:anna}, we recover \cite[Proposition~1.1]{Hamaker.Pechenik.Speyer.Weigandt}, which describes the action of $\nabla$ on Schubert polynomials.
\end{remark}

In order to prove Theorem~\ref{thm:DerivRecursion}, we recall the divided difference operator $\partial_i(f):=\frac{f-s_{i}\cdot f}{x_i-x_{i+1}}$, acting on $\ZZ[x_1, \ldots, x_n]$, and its $K$-analogue $\overline{\partial_i}(f)=\partial_i((1-x_{i+1})f)$. To avoid confusion with the partial derivatives in this appendix, we will write these divided difference operators instead here as $N_i$ and $\overline{N_i}$, respectively.

\begin{lemma}\label{lem:commutation}
The differential operator $\nabla-E$ commutes with  $\overline{N_i}$. 
\end{lemma}
\begin{proof}
It is enough to check that $(\nabla - E) \overline{N_i}(f) = \overline{N_i} (\nabla - E) (f)$ when $f$ is homogeneous. So let $f \in \ZZ[x_1,\ldots, x_n]$ be a  homogenous  polynomial of degree $d$.

By \cite[Lemma~2.1]{Hamaker.Pechenik.Speyer.Weigandt}, $\nabla$ and $N_i$ commute. 
Therefore,  
\begin{align*}
\nabla \overline{N_i}(f)  = \nabla N_i {\Big (}(1-x_{i+1}) f {\Big )} &= N_i  \nabla {\Big (} (1-x_{i+1}) f) {\Big )} \\ &= N_i {\Big (}\! -\!\!f + (1-x_{i+1}) \nabla(f) {\Big )} = - N_i(f) + \overline{N_i}(\nabla f) .
\end{align*}
Note that, if $g$ is homogeneous of degree $e$, then $Eg = eg$, and that $N_i$ lowers degrees by $1$. So we have
\begin{multline*}
E \overline{N_i} (f) = E N_i  {\Big (}(1-x_{i+1}) f {\Big )}  = E N_i f - E N_i x_{i+1} f  = (d -1) N_i(f) -d N_i(x_{i+1} f) \\ = - N_i(f) + d N_i  {\Big (} (1-x_{i+1}) f {\Big )}  = - N_i(f) +d \overline{N_i}(f) = - N_i(f) + \overline{N_i}(Ef). 
\end{multline*}
Subtracting the two equations from each other, 
\[ (\nabla - E)(\overline{N_i} f) = \overline{N_i}(\nabla f)  - \overline{N_i}(Ef) = \overline{N_i}(\nabla - E)f, \]
as required.
\end{proof}

For $w \in S_n$ and for $\{\tau_i\}_{1 \leq i \leq n-1}$ the generators of the $0$-Hecke monoid from Section~\ref{sec:background}, we define
\begin{equation}\label{eq:box+product}
\tau_i \square w = \begin{cases} s_i  w & s_i w <_L w \\ w & s_i w >_L w \\ \end{cases} \qquad 
w \square \tau_i = \begin{cases} w s_i & w s_i <_R w \\ w & w s_i >_R w \\ \end{cases} . 
\end{equation}
The following lemma is easily checked.
\begin{lemma}
The formulas of Equation~\eqref{eq:box+product} define commuting left and right actions of the $0$-Hecke monoid on $S_n$. \qed
\end{lemma}
In this notation, we have $\overline{N_i}(\fG_w) = \fG_{w \square \tau_i}$ by definition.

\begin{proof}[Proof of Theorem~\ref{thm:DerivRecursion}]
By definition, \[
\maj(w^{-1}) = \sum_{s_k w <_L w} k,
\]
so we can rewrite Equation~\eqref{eq:david} as
\begin{equation}\label{eq:oliver}
  (\nabla-E) \fG_w(\bx) = \sum_{s_k w <_L w} k \left( \fG_{s_k w}(\bx)    - \fG_w(\bx) \right)  = \sum_{k=1}^{n-1} k \left( \fG_{\tau_k \square w}(\bx) - \fG_w(\bx) \right). 
  \end{equation}
  
We will prove Equation~\eqref{eq:oliver} by reverse induction on $\inv(w)$. 
We check the base case $w = w_0$. We have $\fG_{w_0}(\bx) = \prod_{i=1}^n x_i^{n-i}$, so
\[ (\nabla - E) \fG_{w_0}(\bx) = \sum_{j=1}^n (n-j) \frac{\prod_{i=1}^n x_i^{n-i}}{x_j}  - \binom{n}{2} \prod_{i=1}^n x_i^{n-i} . \]
On the other hand, $\tau_k \square w_0 = s_k w_0$ and $\fG_{s_k w_0}(\bx) = \frac{\prod_{i=1}^n x_i^{n-i}}{x_{n-k}}$. So the right side of Equation~\eqref{eq:oliver} is
\[  \sum_{k=1}^{n-1} k \left( \frac{\prod_{i=1}^n x_i^{n-i}}{x_{n-k}} - \prod_{i=1}^n x_i^{n-i} \right) = 
\sum_{k=1}^{n-1} k   \frac{\prod_{i=1}^n x_i^{n-i}}{x_{n-k}} - \binom{n}{2} \prod_{i=1}^n x_i^{n-i},\]
and the formulas match.

Now, suppose that we want to prove Equation~\eqref{eq:david} for some $w \neq w_0$, and that we already know Equation~\eqref{eq:david} holds for all longer permutations.
Since $w \neq w_0$, we can find some index $i$ with $w s_i >_R w$ and, by induction, we know that
\[  (\nabla-E) \fG_{w s_i}(\bx)  =  \sum_{k=1}^{n-1} k \left( \fG_{\tau_k \square w s_i }(\bx) - \fG_{w s_i}(\bx) \right). \]
We apply the operator $\overline{N_i}$ to both sides of this equation. On the left, we have by Lemma~\ref{lem:commutation} that
\[ 
\overline{N_i} (\nabla-E) \fG_{w s_i}(\bx) = (\nabla-E) \overline{N_i} \fG_{w s_i}(\bx) = (\nabla - E) \fG_{(w s_i) \square \tau_i}(\bx) = (\nabla - E) \fG_{w}(\bx). 
\]
On the right, we have
\[ 
\sum_{k=1}^{n-1} k \left( \overline{N_i} \fG_{\tau_k \square w s_i }(\bx) -  \overline{N_i} \fG_{w s_i}(\bx) \right) = 
\sum_{k=1}^{n-1} k \left( \fG_{(\tau_k \square w s_i) \square \tau_i }(\bx) -  \fG_{w s_i \square \tau_i}(\bx) \right) .
\]
Since the left and right actions of the $\tau_i$ elements commute, we have \[
(\tau_k \square w s_i) \square \tau_i  = \tau_k \square (w s_i \square \tau_i) = \tau_k \square w,
\]
 so this last formula simplifies to 
\[ 
\sum_{k=1}^{n-1} k \left( \fG_{\tau_k \square w }(\bx) -  \fG_{w}(\bx) \right) ,
\]
which is the required form to complete the induction.
\end{proof}

We now explain the relevance of Theorem~\ref{thm:DerivRecursion} to the results in this paper. 
Let $f$ be a polynomial and let $m$ be a nonnegative integer. If $\deg(f) \neq m$, then $\deg {\big(} (m + \nabla-E)(f) {\big)} = \deg(f)$ 
 whereas, if $\deg(f)=m$, then $\deg {\big(} (m + \nabla-E)(f) {\big)} < \deg(f)$. Thus, equating the degrees of both sides of 
Theorem~\ref{thm:DerivRecursion} implies the following result.
\begin{prop} \label{prop:DerivConsequence}
For all permutations $w \in S_n$,  either
$\deg \fG_w(\bx) = \max_{s_k w <_L w} \deg \fG_{s_k w}(\bx)$ or else
 $\deg \fG_w(\bx) = \maj(w^{-1}) >  \max_{s_k w <_L w} \deg \fG_{s_k w}(\bx)$. \qed
\end{prop}
Recalling Rajchgot's question, we wondered whether Proposition~\ref{prop:DerivConsequence} could be used to inductively compute $\deg \fG_w(\bx)$.
Proposition~\ref{prop:DerivConsequence} determines $\deg \fG_w(\bx)$ when $\maj(w^{-1}) \leq  \max_{s_k w <_L w} \deg \fG_{s_k w}(\bx)$, as then the second case is impossible.
The remaining cases, where $\maj(w^{-1}) >  \max_{s_k w <_L w} \deg \fG_{s_k w}(\bx)$, are the inverse fireworks permutations, and this is how we discovered the fireworks condition.
Proposition~\ref{prop:DerivConsequence} likewise implied $\deg \fG_w(\bx)$ equaled $\maj(u^{-1})$ for some $u \leq_L w$ and, by attempting to describe $u$ in terms of $w$, we found the inverse fireworks map. 
These investigations lead to this paper, although Proposition~\ref{prop:DerivConsequence} does not play a role in our final proofs.

We suspect that the best applications of Theorem~\ref{thm:DerivRecursion} are yet to be found.

\bibliographystyle{amsalpha} 
\bibliography{MM.bib}
\end{document}